\documentclass[sn-mathphys-ay]{sn-jnl}% Math and Physical Sciences Author Year Reference Style
\usepackage{caption} 
\usepackage{graphicx}%
\usepackage{multirow}%
\usepackage{amsmath,amssymb,amsfonts}%
\usepackage{amsthm}%
\usepackage{mathrsfs}%
\usepackage[title]{appendix}%
\usepackage{xcolor}%
\usepackage{textcomp}%
\usepackage{manyfoot}%
\usepackage{booktabs}%
\usepackage{algorithm}%
\usepackage{algorithmic}
\usepackage{listings}%
\usepackage{longtable}  
\renewcommand{\tablename}{TABLE}
\renewcommand{\abstractname}{}
\renewcommand{\thesection}{\arabic{section}.}
\renewcommand{\thesubsection}{\arabic{section}.\arabic{subsection}}
\renewcommand{\thesubsubsection}{\arabic{section}.\arabic{subsection}.\arabic{subsubsection}}

\theoremstyle{thmstyleone}%
\newtheorem{theorem}{Theorem}%  meant for continuous numbers
\newtheorem{lemma}{Lemma}% 

\theoremstyle{thmstyletwo}%

\theoremstyle{thmstylethree}%
\newtheorem{definition}{Definition}%

\begin{document}
\title{\textbf{A Second-Order Algorithm Based on Affine Scaling Interior-Point Methods for nonlinear Optimisation with bound constraints}}

\author*[1]{\fnm{Yonggang} \sur{Pei}}\email{peiyg@htu.edu.cn}

\author[1]{\fnm{Yubing} \sur{Lin}}\email{2401183029@stu.htu.edu.cn}
\equalcont{These authors contributed equally to this work.}

\author*[2]{\fnm{Mauricio Silva} \sur{Louzeiro}}\email{mauriciolouzeiro@ufg.br}
\equalcont{These authors contributed equally to this work.}

\author[3]{\fnm{Detong} \sur{Zhu}}\email{dtzhu@shnu.edu.cn}
\equalcont{These authors contributed equally to this work.}

\affil*[1]{\textit{\orgdiv{School of Mathematics and Statistics}, \orgname{Henan Normal University}, \city{Xinxiang}, \postcode{453007}, China}}
\affil[2]{\textit{\orgdiv{Instituto de Matemática e Estatística}, \orgname{Universidade Federal de Goiás}, \city{Avenida Esperança}, \postcode{74690-900}, Brazil}}
\affil[3]{\textit{\orgdiv{Mathematics and Science College}, \orgname{Shanghai Normal University}, \city{Shanghai}, \postcode{200234}, China}}

\abstract{The homogeneous second-order descent method (Zhang et al. 2025, Mathematics of Operations Research) was initially proposed for unconstrained optimisation problems. HSODM shows excellent performance with respect to the global complexity rate among a certain broad class of second-order methods. In this paper, we extend HSODM to solve nonlinear optimisation problems with bound constraints and propose a second-order algorithm based on affine scaling interior-point methods (SOBASIP).
In each iteration, an appropriate affine matrix is introduced to construct an affine scaling subproblem based on the optimality conditions of the problem. To obtain a valid descent direction similar to HSODM, we utilise the homogenisation technique to transform the scaling subproblem into an Ordinary Homogeneous Model (OHM), which is essentially an eigenvalue problem that can be solved efficiently.
The descent direction is constructed from the optimal solution to the OHM, and then backtracking line search is used to determine the new iteration point. Theoretical analysis establishes that SOBASIP achieves a global iteration complexity of $O(\epsilon^{-3/2})$ for finding an $\epsilon$-approximate second-order stationary point and converges locally at a superlinear rate under certain conditions. Numerical results demonstrate that the proposed method exhibits satisfactory performance.}

\maketitle

\section{Introduction}\label{section1}
In this paper, we focus on the nonlinear minimisation problem with bound constraints
\begin{subequations}\label{f}
\begin{align}
\underset{x \in \mathbf{R}^n}{\text{minimize}} \quad &~ f(x) \\
\text{subject to} \quad &~ l \leq x \leq u,
\end{align}
\end{subequations}
where $ l \in \{ \mathbf{R} \bigcup \{ -\infty \}\}^n$, $ u \in \{ \mathbf{R} \bigcup \{ +\infty \}\}^n$, $l<u$, and $f:\mathbf{R}^n \rightarrow \mathbf{R} $. We denote the feasible set $\mathcal{F} := \{x: l \leq x \leq u\}$ and the strict interior $ \mathrm{int} (\mathcal{F}) := \{x: l < x < u\}$.

Optimisation problems where variables are subjected to bounds are probably one of the most common types of constrained optimisation problems encountered in practical applications \cite{more47,glowinski48,birgin49,ciarlet50,birgin51,glunt52}. 

A variety of algorithms have been developed to solve this type of problem.
%\cite{bjorck5,coleman6,conn7,fletcher8,yang9,more10}.Early methods tended to be of the active set variety \cite{goldfarb39}, in which a sequence of problems are solved.
Early methods tended to be of the active set variety in which a sequence of problems are solved \cite{conn46}.
%In recent papers \cite{birgin40,birgin41,burdakov41}, the efficiency of active-set methods has been emphasized.
However, a significant drawback of such methods for large-scale problems appears to be that the active sets can only change slowly, and many iterations are necessary to correct for a bad initial choice \cite{conn46}.
This motivates a lot of studies on gradient projection methods \cite{goldstein27,levitin28,mccormick29}, which allow the active set to change rapidly from iteration to iteration, and are particularly effective for solving nonlinear programs with bound constraints \cite{birgin41,birgin40,burdakov41}. Nevertheless, the slow convergence rate of the gradient projection method is an obvious drawback \cite{calama53}.
%An efﬁcient technique for solving the nonlinear program with bound constraints is the gradient projection method, proposed  by  Goldstein \cite{goldstein27}, Levitin and Polyak \cite{levitin28}, and more recently, in a less general context, by McCormick \cite{mccormick29}. 
%However, the slow convergence rate of the gradient projection method is an obvious drawback.
%interior-point methods are taken into account in this paper. Interior-point methods are widely recognized as one of the most powerful algorithms for large-scale inequality constrained optimization \cite{forsgren15, nemirovski16, nocedal17}.
%Among these interior-point methods, 
%The affine scaling interior-point method, introduced by Dikin \cite{dikin18} and further developed by Karmarkar in \cite{karmarkar19}, is a different approach to problem \eqref{f}. Furthermore, it is the only interior-point strategy that monotonically decreases the value of the original objective function to reach the solution. 
%In 1994, Thomas F. in \cite{coleman20} proposed an interior-reflective Newton method, where a new affine scaling transformation is introduced, to solve bound-constrained nonlinear minimization problems. Further research efforts focused on this strategy can be found in references \cite{heinkenschloss30,kanzow31,ulbrich32}
%Thomas F. also introduced a trust region and affine scaling interior point method (TRAM) for a general nonlinear minimization with linear inequality constraints in 2000 \cite{coleman21}.
The affine-scaling interior-point method of Coleman and Li \cite{coleman20,coleman21,coleman38} offers an alternative approach to solving problem \eqref{f}. It is based on a reformulation of the necessary optimality conditions obtained by multiplication with a scaling matrix. Notably, without assuming strict complementarity, this method converges superlinearly or quadratically, for a suitable choice of the scaling matrix, when the strong second-order sufficient optimality condition holds \cite{heinkenschloss54}.
Most existing methods for solving bound constrained optimisation problems are limited to finding solutions that satisfy the first-order stationary conditions \cite{birgin40,birgin41,burdakov41,li42,yuan43,facchinei44,yuan45}.
Although first-order methods offer computational efficiency, they have a significant limitation that the critical point they converge to may be a saddle point \cite{zhu34,mokhtari35}.
%In contrast, second-order methods can not only overcome this limitation but also guarantee faster convergence.

Therefore, it is natural to consider higher-order methods that search for the second-order stationary points.
In contrast, second-order methods can not only overcome the drawbacks of first-order methods but also guarantee a faster convergence rate \cite{shao33,zhu34,mokhtari35}.
Nesterov and Polyak \cite{nesterov11} showed the first $O(\epsilon^{-3/2})$ complexity bound of second-order methods by cubic regularisation (CR), which modifies Newton's method by adding a cubic regularisation term to its quadratic approximation model. As an improved version of CR, the adaptive regularisation algorithm using cubics (ARC) proposed by Cartis et al. \cite{cartis12, cartis13}, maintains the same iteration complexity of $O(\epsilon^{-3/2})$ as CR. Moreover, a modified algorithm named trust region algorithm with contractions and expansions (TRACE) \cite{curtis22,curtis23} also realises this complexity by nonlinearly adjusting the trust-region radius. All the methods mentioned above require solving the Newton system, and their computational cost is typically high, making them less efficient. Motivated by the homogenisation trick in quadratic programming \cite{sturm24,he25}, Zhang et al. provided a homogeneous second-order descent method (HSODM) in \cite{zhang14,he37}. This new second-order method converges to an $\epsilon$-approximate second-order stationary point, and is shown to have an $O(\epsilon^{-3/2})$ iteration complexity for nonconvex problems, which is optimal among a certain broad class of second-order methods. 
%In this paper, we aim to extend this method to solve optimization problem \eqref{f}

The affine scaling interior-point method, introduced by Dikin \cite{dikin18} and further developed by Karmarkar in \cite{karmarkar19}, is the only interior-point strategy that monotonically decreases the value of the original objective function to reach the solution. Further research efforts focusing on this strategy for solving nonlinear programs with bound constraints can be found in \cite{heinkenschloss30,kanzow31,ulbrich32}. Given the bound constraints in problem \eqref{f}, the affine scaling interior-point method is considered in this paper.
%Based on this affine scaling framework, Thomas F. in \cite{coleman20} proposed an interior-reflective Newton method for bound-constrained nonlinear minimization, incorporating a novel affine scaling transformation. He further developed an interior trust region approach for the bound-constrained nonlinear minimization problem in \cite{coleman21}. Given the bound constraints in problem \eqref{f}, the affine scaling interior-point method is considered in this paper.
%A key common feature of these two methods \cite{coleman20,coleman21} is their use of an appropriate scaling matrix to construct a quadratic model, eliminating explicit constraint handling. 

Inspired by the above mentioned work, this paper proposes a second-order algorithm based on affine scaling interior-point methods (SOBASIP) for problem \eqref{f}. From the optimality conditions, we introduce a suitable affine matrix and combine it with the Newton step to establish an affine scaling subproblem. The homogenisation technique is then employed to transform the affine scaling subproblem into an Ordinary Homogeneous Model (OHM). In each iteration, we construct the descent direction based on the optimal solution of the OHM, which corresponds to the leftmost eigenvector of the aggregated matrix. Then, combining the backtracking line search method, we show that SOBASIP achieves $O(\epsilon^{-3/2})$ iteration complexity for finding an $\epsilon$-approximate second-order stationary point, and also exhibits local superlinear convergence rate. Although the problem in this paper can be viewed as a special case of that in \cite{he2023}, which develops a Newton-CG based barrier method for finding a second-order stationary point of nonconvex conic optimisation with affine subspace and convex cone constraints, directly applying that framework would lead to an increase in the number of variables. This paper adopts a different analytical framework while maintaining the same $O(\epsilon^{-3/2})$ iteration complexity.

Our main contributions are as follows.
\begin{itemize}
\item We propose a novel second-order algorithm based on affine scaling interior-point methods, which is an extended application of the HSODM in \cite{zhang14} to solve bound constrained optimisation problems. This provides a new perspective for addressing such problems.
\item Under mild assumptions, we establish that the proposed method achieves an iteration complexity of $O(\epsilon^{-3/2})$ to find an $\epsilon$-approximate second-order stationary point, which is a competitive complexity bound for bound-constrained optimisation. Moreover, unlike many second-order approaches that require solving expensive Newton systems, our approach only needs to solve an eigenvalue minimisation problem per iteration. 
\item Under appropriate assumptions, the algorithm SOBASIP exhibits a local superlinear convergence rate with parameter $\delta=0$, which is a very desirable result for this class of problems.
\end{itemize} 
%He and Lu \cite{he2023} proposed a Newton-CG based barrier method for finding a second-order stationary point of nonconvex conic optimization with affine subspace and convex cone constraints. It gives a positive answer to the open question by O’Neill and Wright \cite{o2021}, extending the log-barrier Newton-CG method for bound constraints to general conic constraints. Although the problem considered in our paper can be viewed as a special case of that in \cite{he2023} and can be reformulated into the same form, such a reformulation will change the original structural properties of the problem. Therefore, we adopt a different analytical framework. Despite the different analysis approaches, both methods achieve the same iteration complexity $O(\epsilon^{-3/2})$ with guaranteed operation complexity.
%In addition, the research status of existing methods regarding the approximate second-order stationary points for nonconvex conic optimization problems were also reviewed in \cite{he2023} (see the references cited therein).
%has a global convergence rate of $O(\epsilon^{-3/2})$ to find an $\epsilon$-approximate second-order stationary point. Consequently, SOBASIP  achieves a local superlinear convergence rate by setting the perturbation parameter $\delta=0$ under appropriate assumptions.

The rest of this paper is organised as follows. In Section \ref{section2}, we construct OHM based on the affine scaling subproblem, then solve it as an eigenvalue problem, and introduce the corresponding SOBASIP in Algorithm \ref{SOBASIP}. Some preliminary results are presented in Section \ref{section3}. In Section \ref{section4} and \ref{section5}, we provide an analysis of the global and local convergence of SOBASIP. The result indicates that SOBASIP attains an iteration complexity of $O(\epsilon^{-3/2})$ when seeking an $\epsilon$-approximate second-order stationary point. If the algorithm does not early terminate, it converges at a local superlinear convergence rate. In Section \ref{section6}, numerical results of Algorithm \ref{SOBASIP} are provided.
\subsection{Notations}\label{1.1}
In this subsection, we introduce the notations throughout this paper.

Let $\| \cdot \|$ be the standard Euclidean norm on $\mathbf{R}^n$. The closed ball centered at $x$ with radius $r$ is denoted 
$B(x,r) = \left\{ y \in \mathbb{R}^n \mid \|y - x\| \leq r \right\}$. Given a matrix 
$A \in \mathbf{S}^{n \times n}$, $\|A\|$ stands for its induced $l_2$-norm, while 
$\lambda_1(A)$, $\lambda_2(A), \dots, \lambda_{\text{max}}(A)$ denote the distinct eigenvalues of $A$ arranged in ascending order.
For the vector-valued function $h: \mathbf{R}^n \rightarrow \mathbf{R}^n$, we define the i-th component of 
$h$ as $h^i$. $ \lceil \cdot \rceil$ in this paper denotes the ceiling function.
We employ the standard order notations $O$, $\Omega$, and $\Theta$ as they are commonly used. Specifically, for two constants $A$ and $B$, we state that $A = O(B)$ if there exists a positive constant $c$ such that $A \leq c \cdot B$. Similarly, $A = \Omega(B)$ if there exists a positive constant $c$ where $A \geq c \cdot B$. We define $A = \Theta(B)$ when both $A = O(B)$ and $A = \Omega(B)$ hold. The notation $[a;b]$ is used to represent the vertical concatenation of arrays or numbers.

\section{Development of the algorithm }\label{section2}
%In this paper, we focus on the bound-constrained optimization problems of the form 
%\begin{subequations}\label{f}
%\begin{align}
%\underset{x \in \mathbf{R}^n}{\text{minimize}} \quad &~ f(x) \\
%\text{subject to} \quad &~ l \leq x \leq u,
%\end{align}
%\end{subequations}
%where $ l \in \{ \mathbf{R}^n \bigcup \{ -\infty \}\}^n$, $ u \in \{ \mathbf{R}^n \bigcup \{ +\infty \}\}^n$, $l<u$, and $f:\mathbf{R}^n \rightarrow \mathbf{R} $. We denote the feasible set $\mathcal{F} := \{x: l \leq x \leq u\}$ and the strict interior $ \mathrm{int} (\mathcal{F}) := \{x: l < x < u\}$.
In this section, we first establish OHM from the affine scaling subproblem, with the ideas presented in \cite{zhang14}. We then elaborate on the core idea of the proposed algorithm in further detail, and conclude this section by providing a comprehensive formal description of the entire algorithm.

We use $g(x)$ and $H(x)$ to represent the gradient and Hessian matrix of the objective function $f(x)$, i.e., $g(x) := \nabla f(x), H(x):={\nabla}^2 f(x)$. Constraints make it difficult to formulate a similar subproblem for which a global solution can be computed by existing software. However, the difficulties caused by constraints in problem \eqref{f} can be addressed by applying affine scaling to construct an approximate quadratic function. Thus, we first define a vector function $ v(x): \mathbf{R}^n \rightarrow \mathbf{R}^n$ as follows.
\begin{definition}\label{definition1}
The vector $v(x)=(v^1(x), v^2(x),..., v^i(x)) \in \mathbf{R}^n$ is defined as follows. For each component $1 \leq i \leq n$,
\begin{equation}
v^i(x):=
\left\{
\begin{aligned}
&x^i - u^i, & \text{ if } g^i(x)<0 \text{ and } u^i<\infty,\\
&x^i - l^i, & \text{ if } g^i(x)\geq 0 \text{ and }l^i>-\infty, \\
&-1, & \text{ if } g^i(x)<0 \text{ and } u^i=\infty\\
&1, &\text{ if } g^i(x)\geq 0 \text{ and } l^i=-\infty.
\end{aligned}
\right.
\end{equation}
\end{definition}
We remark that for any $s\in \mathbf{R}^n$, $\mathrm{diag}(s)$ denotes an $n$-by-$n$ diagonal matrix, and the diagonal entries of this matrix are the components of vector $s$ in their original sequence. Using this notation, we define an affine scaling matrix
\begin{equation}\label{D}
  D(x) := \mathrm{diag}(|v(x)|^{-\frac{1}{2}}),
\end{equation}
i.e., $D^{-2}(x)$ is a diagonal matrix with the $i$-th diagonal component equal to $|v^i(x)|$.

Assuming feasibility and $g_{\ast}:=g(x_{\ast})$, first-order necessary optimality conditions for $x_{\ast}$ to be a local minimizer are
\begin{equation}\label{first-order}
\left\{
\begin{aligned}
g_{\ast}^i = 0, & \text{ if } l^i<x_{\ast}^i<u^i,\\
g_{\ast}^i \leq 0, & \text{ if } x_{\ast}^i=u^i, \\
g_{\ast}^i \geq 0, & \text{ if } x_{\ast}^i=l^i.
\end{aligned}
\right.
\end{equation}
Similar to what in \cite{coleman20,coleman38}, we consider the following system
\begin{equation}\label{sys}
  D(x)^{-2}g(x)=0.
\end{equation}
It is clear that system \eqref{sys} is equivalent to the first-order necessary conditions we referred to previously. System \eqref{sys} is continuous but not everywhere differentiable. Nondifferentiability occurs when $v^i=0$. We may prevent this situation from occurring by restricting $x_k \in \mathrm{int} (\mathcal{F})$. However, if $g^i(x)=0$, $v^i$ is discontinuous on such a point but $D(x)^{-2}g(x)=0$ is continuous. Moreover, Coleman and Li show that it is possible to generate a second-order Newton process for system \eqref{sys} in \cite{coleman20}.

Assume $x_k\in \mathrm{int} (\mathcal{F})$, the Newton step of \eqref{sys} satisfies
\begin{equation}\label{niudunbu}
\left(D^{-2}_kH_k+\operatorname{diag}(g_k) J^v_k\right)d_k=-D^{-2}_kg_k
\end{equation}
where $g_k:=g(x_k),H_k:=H(x_k), D_k:=D(x_k),$ and $J^v_k \in \mathbf{R}^{n\times n} $ is the Jacobian matrix of $|v(x)|$ at $x_k$ whenever $|v(x)|$ is differentiable.

Multiplying the two sides of \eqref{niudunbu} by $D_k$ and defining
\begin{equation}\label{fangshedk}
	\left\{
	\begin{aligned}
		&\overline{d}_k:=D_kd_k,\\
		&\overline{H}_k:=D^{-1}_kH_kD^{-1}_k,\\
		&\overline{C}_k:=\operatorname{diag}(g_k)J_k^v,
	\end{aligned}
	\right.
\end{equation}
we have that
\begin{equation*}\label{}
\left(D^{-1}_kH_kD^{-1}_k+\operatorname{diag}(g_k) J^v_k\right)\overline{d}_k=-D^{-1}_kg_k,
\end{equation*}
which is equivalent to the following system
\begin{equation*}\label{}
\left(\overline{H}_k+\overline{C}_k\right)\overline{d}_k=-\overline{g}_k,
\end{equation*}
where $\overline{g}_k:=D^{-1}_kg_k$.
Then, based on the above system, we can establish an affine scaling subproblem
\begin{equation}\label{ziwenti}
\underset{\overline{d} \in \mathbf{R}^n}{\text{minimize   }} 
\overline{g}_k^T\overline{d}+\frac{1}{2}\overline{d}^T\overline{B}_k\overline{d},
\end{equation}
where $\overline{B}_k$ is $\overline{H}_k+\overline{C}_k$.

Inspired by \cite{zhang14}, we attempt to transform the affine scaling subproblem \eqref{ziwenti} into a homogeneous quadratic model via homogenisation trick. Let $\overline{d}:=s/t$,  
\begin{align}\label{qicihua}
t^2\left(\overline{g}_k^T({s/t}) + \frac{1}{2} ({s/t})^T\overline{B}_k({s/t}) - \frac{1}{2} \delta\right)
&=t\cdot \overline{g}_k^Ts + \frac{1}{2} s^T\overline{B}_k s - \frac{1}{2}{\delta} t^2
=\frac{1}{2} \begin{bmatrix}
                      s \\ t
                  \end{bmatrix}^T
F_k
    \begin{bmatrix}
        s \\ t
    \end{bmatrix},
\end{align}
where
\begin{align}\label{F_k}
 & F_k := \begin{bmatrix}
                                     \overline{B}_k   & \overline{g}_k    \\
                                     \overline{g}_k^T & -\delta
                                 \end{bmatrix}.
\end{align}                                 
The second equation is what we refer to as the homogeneous quadratic model. Similar to \cite{zhang14}, a ball constraint $\|[s;t]\|\le 1$ is added to model \eqref{qicihua}. Moreover, by setting $\overline{d}:=s/t$, the homogenised model becomes equivalent to the $t^2$ scaled version of the affine scaling subproblem \eqref{ziwenti}, where the equivalence is valid up to a constant.

Given an iterate $x_k \in \mathbb{R}^n$, we define the Ordinary Homogeneous Model (OHM) as follows:
\begin{equation}
    \label{hsodm.subproblem}
    \begin{aligned}
        \min_{\|[s; t]\| \le 1}
        \psi_k(s, t; \delta) := ~ \begin{bmatrix}
        s \\ t
    \end{bmatrix}^T
    \begin{bmatrix}
        \overline{B}_k   & \overline{g}_k    \\
        \overline{g}_k^T & -\delta
    \end{bmatrix}
    \begin{bmatrix}
        s \\ t
    \end{bmatrix},~ s\in \mathbb{R}^n, t \in \mathbb{R},
    \end{aligned}
\end{equation}
where $\delta \geq 0$ is a predefined constant.

Denote the optimal solution of problem \eqref{hsodm.subproblem} as $[s_k; t_k]$. In fact, subproblem \eqref{hsodm.subproblem} is an eigenvalue problem in which $[s_k; t_k]$ is the eigenvector corresponding to the smallest eigenvalue of $F_k$. Therefore, we can solve this subproblem by using eigenvector finding procedure.
 %see \cite{carmon_accelerated_2018, kuczynski_estimating_1992, royer_complexity_2018}.

After solving \eqref{hsodm.subproblem}, a descent direction $d_k$ is determined by $d_k=D^{-1}_k \overline{d}_k$, where $\overline{d}_k$ is constructed based on the optimal solution $[s_k; t_k]$. Evidently, $\overline{d}_k=s_k/t_k$ is a wish choice. However, when $t_k=0$, an undesirable situation arises where $d_k$ tends to infinity. To prevent this scenario, we choose the truncated direction $s_k$ as the descent direction if $|t_k|$ is sufficiently small. The reason is that matrix $\overline{B}_k$ dominates the homogenised model when $t_k$ is small. Otherwise, we directly choose $s_k/t_k$ as the descent direction because predefined constant $\delta$ becomes important in this case. In this paper, we use$\sqrt{1/(1+\Delta^2)}$ and $\nu$ as criteria to judge whether $t_k$ is very small, where $\Delta$ and $\nu$ are predefined constants of the algorithm. Once the descent direction has been determined, we employ the backtracking line search method to guarantee an sufficient decrease of the model.

However, to achieve a balance between advancing along the search direction $d_k$ for objective function optimisation and maintaining the iterate points within the feasible region, before starting the backtracking line search, we define
\begin{equation}\label{maxstepsize}
\alpha_k^{\text{max}} := \min \bigg\{\max\bigg\{ \frac{l_i-x_k^i}{d_k^i}, \frac{u_i - x_k^i}{d_k^i}\bigg\}:1\leq i\leq n \bigg\}.
\end{equation}

For the line search, we utilise a backtracking subroutine to determine the stepsize $\alpha_k$, ensuring it produces a sufficient decrease. Here, a basic sufficient decrease requirement for $\alpha_k\in(0,1]$ is given as follows
\begin{equation}\label{decrease condition}
f(x_k+\alpha_k d_k)-f(x_k) \leq -\frac{\gamma}{6}{\alpha}_k^3 \|\overline{d}_k\|^3.
\end{equation}

Based on the above discussion, we give the whole algorithm as follows.
\begin{algorithm}
\caption{SOBASIP}
\label{SOBASIP}
\begin{algorithmic}
\REQUIRE
  Initial iterate $x_0 \in \mathrm{int} (\mathcal{F})$, 
  $\nu \in (0,1/2)$, $\gamma>0$, $\beta \in (0,1)$, $\tau \in (0,1)$, $\Delta=\Theta(\sqrt{\epsilon})$ and 
  $\epsilon>0$ is sufficiently small.

\FOR{$k=1,2,\dots$} 
  \STATE Compute $f_k$, $g_k$, $D_k^{-1}$, $H_k$, $B_k$.
  
  % Step 2: 求解子问题
  \STATE Compute $[s_k; t_k]$ by solving subproblem \eqref{hsodm.subproblem}.
  
  \IF{$|t_k| > \sqrt{\frac{1}{1+\Delta^2}}$}
    \STATE Set $\overline{d}_k = s_k / t_k$.
    \STATE Compute $d_k=D_k^{-1}\overline{d}_k $.
    \STATE Set $x_{k+1} = x_k + d_k$.
    \STATE \textbf{return} $x_{k+1}$
  \ENDIF
    \IF{$|t_k| \geq \nu$}
        \STATE Set $\overline{d}_k = s_k / t_k$.
    \ELSE
        \STATE Set $\overline{d}_k = \text{sign}(-g_k^T s_k) \cdot s_k$.
    \ENDIF
    \STATE Compute $d_k=D_k^{-1}\overline{d}_k$.
      
    % 线搜索
    \STATE Compute $\alpha_k^{\mathrm{max}}$ by \eqref{maxstepsize}.
    \STATE Set $j=0$, $\alpha_{k,0} =\tau \cdot \min\left\{1, \alpha_k^{\mathrm{max}} \right\}$.
    \WHILE{true}
    \STATE Compute $\Delta f = f(x_k + \alpha_{k,j} d_k) - f(x_k)$.
    \IF{$\Delta f \leq -\frac{\gamma}{6}{\alpha}_{k,j}^3 \|\overline{d}_k\|^3$} % 满足条件：退出循环
        \STATE Set $\alpha_k = \alpha_{k,j}$.
        \STATE \textbf{break} 
    \ENDIF
    \STATE Set $\alpha_{k,j+1} = \beta \cdot \alpha_{k,j}$. % 不满足：衰减步长
    \STATE Set $j = j + 1$.
    \ENDWHILE
  
  % 线搜索满足后：更新迭代点
  \STATE $x_{k+1} = x_k + \alpha_k d_k$.
\ENDFOR 
\end{algorithmic}
\end{algorithm}

%\begin{algorithm}
%\caption{Homogeneous Second-Order Descent Method (HSODM)}
%\label{alg:hsodm}
%\textbf{Data:} initial point $x_1$, $\nu \in (0, 1/2)$, $\Delta = \Theta(\sqrt{\epsilon})$
%\begin{algorithmic}[1] % [1] 表示显示行号
%\For{$k = 1, 2, \cdots$}
%    \State Solve the subproblem \eqref{hsodm.subproblem}, and obtain the solution $[v_k; t_k]$;
%    \If{$|t_k| > \sqrt{1/(1 + \Delta^2)}$} \Comment{small value case}
%        \State $\overline{d}_k \gets v_k / t_k$;
%        \State Update $x_{k+1} \gets x_k + \overline{d}_k$;
%        \State (Early) Terminate (or set $\delta = 0$ and proceed);
%    \EndIf
%    \If{$|t_k| \geq \nu$} \Comment{large value case (a)}
%        \State $\overline{d}_k \gets v_k / t_k$;
%    \Else \Comment{large value case (b)}
%        \State $\overline{d}_k \gets \text{sign}(-\overline{g}_k^T v_k) \cdot v_k$;
%    \EndIf
%    \State Choose a stepsize $\alpha_k$ by the fixed-radius strategy or the line search strategy (see Algorithm 2);
%    \State Update $x_{k+1} \gets x_k + \alpha_k \cdot \overline{d}_k$;
%\EndFor
%\end{algorithmic}
%\end{algorithm}

\section{Preliminary Results}\label{section3}
Before establish the global convergence, we present some preliminary analysis of the homogenised quadratic model. First, we analyse the relationship between the smallest eigenvalues of the $\overline{B}_k$ and the aggregated matrix $F_k$, and the perturbation parameter $\delta$ in the following Lemma.
\begin{lemma}[Relationship between $\lambda_1(F_k)$, $\lambda_1(\overline{B}_k)$ and $\delta$]
    \label{lemma1}
    Let $\lambda_1(\overline{B}_k)$ and $\lambda_1(F_k)$ be the smallest eigenvalue of $\overline{B}_k$ and $F_k$ respectively. If $\overline{g}_k\neq 0$ and $\overline{B}_k \neq 0$, then 
    $\lambda_1(F_k) < -\delta$ and $\lambda_1(F_k) \leq \lambda_1(\overline{B}_k)$ hold.
\end{lemma}
\begin{proof}
From the definition of $F_k$ in \eqref{F_k} and the Cauchy interlace theorem in \cite{Parlett3}, we can immediately obtain $\lambda_1(F_k)  \leq \lambda_1(\overline{B}_k) $.
Then we prove $\lambda_1(F_k) < -\delta$, which is equivalent to proving that the matrix $F_k+\delta I$ has a negative eigenvalue.

Let us consider the direction $\left[ -\zeta \overline{g}_k ; t \right]$, where $\zeta,t>0$. Define the following function of $\left( \zeta,t \right)$:
    \begin{equation*}
        \begin{aligned}
            f(\zeta,t)  :=& \begin{bmatrix}
                               -\zeta \overline{g}_k \\
                               t
                           \end{bmatrix}^T 
                           (F_k+\delta I)
                            \begin{bmatrix}
                            -\zeta \overline{g}_k \\
                             t
                             \end{bmatrix},     \\
                         =&\begin{bmatrix}
                               -\zeta \overline{g}_k \\
                               t
                           \end{bmatrix}^T 
                         \begin{bmatrix}
                             \overline{B}_k + \theta_k \cdot I  & \overline{g}_k  \\
                             \overline{g}_k^T        & -\delta+\theta_k
                            \end{bmatrix}  
                         \begin{bmatrix}
                             -\zeta \overline{g}_k \\
                              t
                         \end{bmatrix},     \\                                    
                      =&\zeta^2 \overline{g}_k^T (\overline{B}_k+\delta I) \overline{g}_k -2\zeta t \|\overline{g}_k\|^2.
        \end{aligned}
    \end{equation*}
    For any fixed $t>0$, we have
    \begin{equation*}
        f(0,t)=0 \quad  \text{and} \quad  \frac{\partial f(0,t)}{\partial \zeta}=-2t \|\overline{g}_k\|^2<0.
    \end{equation*}
    Therefore, for sufficiently small $\zeta>0$, it holds that $f(\zeta,t)<0$, which implies that $\left[ -\zeta \overline{g}_k ; t \right]$ is a negative curvature. Hence, we prove $\lambda_1(F_k)+\delta<0 $. 
\end{proof}

In the following lemma, we characterise the optimal solution $[s_k; t_k]$ of problem \eqref{hsodm.subproblem} based on the optimality condition of the standard trust-region subproblem.
\begin{lemma}[Optimality condition]\label{lemma2}
    $[s_k; t_k]$ is the optimal solution of the subproblem \eqref{hsodm.subproblem} if and only if there exists a dual variable $\theta_k > \delta \geq 0$ such that
\begin{subequations}
    \begin{align}
        \label{lemma2.1}
                                    & \begin{bmatrix}
                                          \overline{B}_k + \theta_k \cdot I & \overline{g}_k              \\
                                          \overline{g}_k^T                  & -\delta+\theta_k
                                      \end{bmatrix} \succeq 0, \\
        \label{lemma2.2}
                                    & \begin{bmatrix}
                                          \overline{B}_k + \theta_k \cdot I & \overline{g}_k              \\
                                          \overline{g}_k^T                  & -\delta+\theta_k
                                      \end{bmatrix}
        \begin{bmatrix}
            s_k \\ t_k
        \end{bmatrix} = 0,                                                      \\
        \label{lemma2.3} & \|[s_k; t_k]\| = 1.
    \end{align}
\end{subequations}
    Moreover, $-\theta_k$ is the smallest eigenvalue of the perturbed homogenised matrix $F_k$, i.e., $-\theta_k = \lambda_1(F_k)$.
\end{lemma}
\begin{proof}
    By the optimality condition of the standard trust-region subproblem, $[s_k; t_k]$ is the optimal solution if and only if there exists a dual variable $\theta_k \geq 0$ such that
    \begin{equation*}
        \begin{bmatrix}
            \overline{B}_k + \theta_k \cdot I & \overline{g}_k             \\
            \overline{g}_k^T                  & -\delta+\theta_k
        \end{bmatrix} \succeq 0, \ \begin{bmatrix}
            \overline{B}_k + \theta_k \cdot I & \overline{g}_k              \\
            \overline{g}_k^T                  & -\delta+\theta_k
        \end{bmatrix} \begin{bmatrix}
            s_k \\ t_k
        \end{bmatrix} = 0, \ \text{and} \  \theta_k \cdot (\|[s_k; t_k]\| - 1) = 0.
    \end{equation*}
    With Lemma \ref{lemma1}, we have $\lambda_1(F_k) < -\delta \leq 0$.
    Therefore, $\theta_k \geq  - \lambda_1(F_k) >  \delta \geq 0$, and further $ \|[s_k; t_k]\| = 1 $. Moreover, by \eqref{lemma2.2}, we obtain
    \begin{equation*}
        \begin{bmatrix}
            \overline{B}_k   & \overline{g}_k     \\
            \overline{g}_k^T  & -\delta
        \end{bmatrix} \begin{bmatrix}
            s_k \\ t_k
        \end{bmatrix} = -\theta_k \begin{bmatrix}
            s_k \\ t_k
        \end{bmatrix}.
    \end{equation*}
    Multiplying the equation above by $ \left[ s_k ; t_k \right]^T $, we have
    \begin{equation*}
        \min_{\|[s; t]\| \le 1} \psi_k(s, t; \delta) = -\theta_k.
    \end{equation*}
    Note that with \eqref{lemma2.3}, the optimal value of problem \eqref{hsodm.subproblem} is equivalent to the smallest eigenvalue of $F_k$, i.e., $\lambda_1(F_k)$. Thus, $-\theta_k = \lambda_1(F_k)$.
\end{proof}

Based on the above optimality condition, we can derive the following lemma.
\begin{lemma}
    \label{lemma3}
    The equation \eqref{lemma2.2} in Lemma \ref{lemma2} can be rewritten as,
    \begin{equation}\label{lemma3.1}
        \left(\overline{B}_k + \theta_k I\right)  s_k =  - t_k \overline{g}_k \quad \text{and} \quad \overline{g}_k^T s_k  =  t_k (\delta - \theta_k).
    \end{equation}
    Furthermore,
    \begin{enumerate}
        \item[(1)] If $t_k = 0$, then we have
            \begin{equation}\label{lemma3.2}
                \left(\overline{B}_k + \theta_k I\right) s_k = 0 \quad \text{and} \quad \overline{g}_k^T s_k  = 0,
            \end{equation}
            where the first equality implies that $(-\theta_k,s_k)$ is the eigenpair of the matrix $\overline{B}_k$.
        \item[(2)] If $t_k \neq 0$, then we have
            \begin{equation}
                \label{lemma3.3}
                \overline{g}_k^T \overline{d}_k = \delta -\theta_k \quad \text{and} \quad \left(\overline{B}_k+\theta_k \cdot I\right)\overline{d}_k =-\overline{g}_k
            \end{equation}
            where $\overline{d}_k =s_k / t_k$.
    \end{enumerate}
\end{lemma}
As the lemma above is a direct application of Lemma \ref{lemma2}, we do not present its proof here.

\begin{lemma}[Nontriviality of direction $s_k$]\label{lemma4}
    If $\overline{g}_k \neq 0$, then $s_k \neq 0$.
\end{lemma}
\begin{proof}
    We prove this by contradiction. Suppose that $s_k = 0$. Then, we have $t_k \overline{g}_k = 0$ with equation \eqref{lemma3.1} in Lemma \ref{lemma3}. It further implies that $t_k = 0$ due to $g_k \neq 0$. However, $[ s_k ; t_k]=0$ contradicts to the equation $\|[s_k; t_k]\| = 1$ in the optimality condition. Therefore, we have $s_k \neq 0$.
\end{proof}

This lemma establishes the existence of a nontrivial direction $v_k$ at all times, ensuring that Algorithm \ref{SOBASIP} will never get stuck.

\begin{lemma}\label{lemma5}
    For the sign function value $\text{sign}(-\overline{g}_k^T s_k)$, we always have $\text{sign}(-\overline{g}_k^T s_k) \cdot t_k = |t_k|$.
\end{lemma}
\begin{proof}
    By the second equation of optimal condition \eqref{lemma3.1} and the fact that $\delta < \theta_k$, we obtain that
    $$
        \text{sign}(-\overline{g}_k^T s_k) = \text{sign}(t_k),
    $$
    and it implies
    $$
        \text{sign}(-\overline{g}_k^T s_k) \cdot t_k = \text{sign}(t_k) \cdot t_k = |t_k|.
    $$
    This completes the proof.
\end{proof}

\begin{lemma}[Trivial case, $\overline{g}_k=0$]\label{lemma6.add}
Suppose that $\overline{g}_k$=0, then the following statements hold.
\begin{enumerate}
\item[(1)] If $\lambda_1(\overline{B}_k) > -\delta$, then $t_k=1$.
\item[(2)] If $\lambda_1(\overline{B}_k) < -\delta$, then $t_k=0$.
\end{enumerate}
\end{lemma}
\begin{proof}
When $\overline{g}_k=0$, we can immediately get
$$ \min_{\|[s; t]\| \le 1}
        \psi_k(s, t; \delta)=s^T\overline{B}_ks-t^2 \cdot \delta.$$
We first prove the statement (1) by contradiction, and the other one can be proved by the same argument. Suppose that $t_k\neq 1$, then we have $s_k\neq0$ with the help of equation \eqref{lemma2.3}. Thus,
$$\psi_k(s_k, t_k; \delta)=s^T\overline{B}_ks-t^2 \cdot \delta > -\delta=\psi_k(0, 1; \delta),$$
where the inequality holds due to $s^T\overline{B}_ks \geq \lambda_1(\overline{B}_k)\|s_k\|^2 > -\delta\|s_k\|^2$. The above inequality contradicts to the optimality of $(s_k,t_k)$, and thus $t_k=1$.
\end{proof}
\section{Global convergence}\label{section4}
In this section, we analyse the convergence rate of the Algorithm \ref{SOBASIP}. We consider the large and small values of $\|\overline{d}_k\|$, respectively.
For the large value case of $\|\overline{d}_k\|$, we show that the function value decreases by at least $\Omega(\epsilon^{3/2})$ at every iteration after carefully selecting the perturbation parameter $\delta$. In the latter case, we prove that the next iterate $x_{k+1}$ is already an $\epsilon$-approximate SOSP, and thus the algorithm can terminate. Throughout the paper, we make the following assumptions.\\
%\noindent \(\mathbf{(A1)}\) \quad  The iterates $\{x_k\}$ which is the sequence generated by Algorithm \ref{SOBASIP}, remain in a closed, bounded domain  $X\subset\mathbb{R}^n$.\\
\noindent \(\mathbf{(A1)}\) \quad  There exists a bounded closed set $X \subset \mathrm{int} (\mathcal{F})$ such that all iterates $\{x_k\}$ generated by the algorithm satisfy $\{x_k\} \subset X$.\\
\noindent \(\mathbf{(A2)}\) \quad  $f(x)$ is thrice continuously differentiable on $X$. $H(x)$ is Lipschitz continuous on $X$, i.e., there exists a constant $L_H>0$ such that
\begin{equation}\label{H.lipschitz}
        \|H(x) - H(y)\| \le L_H \|x-y\|, ~\forall x, y \in X,
\end{equation}
and that the Hessian matrix is bounded,
\begin{equation}\label{bound.Hk}
\|H_k\| \leq U_H, ~\forall k \geq 0,
\end{equation}
for some $U_H > 0$.
\\
\noindent \(\mathbf{(A3)}\) \quad $\{D_k^{-1}\}$ and $\{\overline{B}_k\}$ are bounded, i.e., there exist constants $\gamma_D>0$, $\gamma_B>0$ such that 
\begin{equation}\label{boundB.D}
\|D_k^{-1}\| \leq \gamma_D \quad \text{and} \quad \|\overline{B}_k\| \leq \gamma_B
\end{equation}
for all $k>0$.
\\

(A1) also implies that $\{D_k^{-1}\}$ is uniformly continuous on the sequence of iterates $\{x_k\}$, i.e., 
\begin{equation}\label{uniformly continuous}
\|D_{r_i}^{-1}-D_{l_i}^{-1}\| \to 0
\end{equation}
whenever $\|x_{r_i}-x_{l_i}\| \to 0, i \to \infty$, where $\{x_{r_i}\}$ and $\{x_{l_i}\}$ are the subsequences of $\{x_k\}$.\\

\begin{definition}\label{definition2}
A point $x$ is called an $\epsilon$-approximate second-order stationary point (SOSP) if it satisfies the following conditions:
\begin{subequations}
 \begin{align} 
     & \|\overline{g}(x)\|\leq O(\epsilon), \label{def1.1}                        \\
     & \lambda_1 (\overline{B}(x)) \geq \Omega(-\sqrt{\epsilon}). \label{def1.2}
 \end{align}
\end{subequations}
\end{definition}
Lemma 1.2.4 in \cite{Nesterov4} provides the following result, which plays an important role in proving the function decrease value and the convergence of the Algorithm \ref{SOBASIP}.
\begin{lemma}\label{lemma6}
Suppose that assumptions (A1)-(A2) hold. Then for all $x, y \in \mathbb{R}^n$, we have
\begin{subequations}
        \begin{align}
             & \left\|g(y)-g(x)-H(x)(y-x)\right\|                       \leq \frac{L_H}{2}\|y-x\|^{2}, \label{lemma6.1a} \\
             & \left|f(y)-f(x)-g(x)^T(y-x)-\frac{1}{2}(y-x)^TH(x)(y-x)\right|  \leq \frac{L_H}{6}\|y-x\|^{3}. \label{lemma6.1b}
        \end{align}
    \end{subequations}
\end{lemma}

\subsection{Analysis for the scenario of large value}\label{section4.1}
As discussed in \cite{zhang14}, we also define the large value case of   $\|\overline{d}_k\|$ as the case that its norm is larger than $\Delta$ which is a constant, i.e., $\|\overline{d}_k\| > \Delta$. It is clear that whether the case $\nu \le |t_k| \le \sqrt{1/(1+\Delta^2)}$ or the case $|t_k| \le \nu$ with $\nu \in (0, 1/2)$, $\|\overline{d}_k\| > \Delta$ will always be obtained. Therefore, we call these two cases the large value case. In this scenario, the homogenised direction can be either $\overline{d}_k = \text{sign}(-\overline{g}_k^T s_k) \cdot s_k$ or $\overline{d}_k = s_k / t_k$. The following discussion shows that the function value decreases by at least $\Omega(\epsilon^{3/2})$ in the scenario of large value.

%Before discussing the decrease of the function value, we first state that
%\begin{equation}\label{banzhengding}
%\overline{d}_k^T \overline{C}_k \overline{d}_k \geq 0.
%\end{equation}
From the definition of the $\overline{C}_k$ in \eqref{fangshedk}, we have that
\begin{equation}\label{Ck}
\overline{C}_k = \operatorname{diag}(g_k)J_k^v=\operatorname{diag}\bigg(\frac{\partial \left|v^i(x_k)\right|}{\partial x_k^i} \cdot g_k^i\bigg),
\end{equation}
where 
\begin{equation*}
\frac{\partial \left|v^i(x_k)\right|}{\partial x_k^i}:=
\left\{
\begin{aligned}
&1,  & \text{ if } g_k^i\geq 0 \text{ and }l^i>-\infty,\\
&0,  & \text{ if } g_k^i<0 \text{ and } u^i=\infty, \\
&0,  & \text{ if } g_k^i\geq 0 \text{ and } l^i=-\infty,\\
&-1, & \text{ if } g_k^i<0 \text{ and } u^i<\infty.\\
\end{aligned}
\right.
\end{equation*}
It is obvious that $\overline{C}_k$ is a positive semidefinite diagonal matrix.
%implying that $\overline{d}_k^T \overline{C}_k \overline{d}_k \geq 0$ holds.

We derive the descent lemma using the line-search strategy and further establish an upper bound on the number of iterations needed for the backtracking line search procedure. For the cases $|t_k| < \nu$ and $|t_k| \geq \nu$, we obtain the following two lemmas that characterise the sufficient decrease property.
\begin{lemma}\label{lemma7}
Suppose that assumptions (A1)-(A3) hold and set $\nu \in (0, 1/2)$. If $|t_k| < \nu$, then let $\overline{d}_k = \text{sign}(-\overline{g}_k^T s_k) \cdot s_k$. The backtracking line search terminates with $\alpha_{k,j}=\beta^{j_k}\alpha_{k,0}$,  and $j_k$ is upper bounded by
    $$
        j_N := \left \lceil \log_\beta\left(\frac{3\delta}{L_H\gamma_D^3+\gamma}\right) \right\rceil,
    $$and the function value associated with the stepsize $\alpha_k$ satisfies,
    \begin{equation}\label{decrease1}
        f(x_{k+1}) - f(x_k) \leq -\min \left\{\frac{\sqrt{3}\gamma}{16}\underline{\alpha}^3, \frac{9\gamma\beta^3\delta^3}{2\left(L_H\gamma_D^3+\gamma\right)^3}\right\}.
    \end{equation}
\end{lemma}
\begin{proof}
%If the backtracking line search terminates with $\alpha_k=\alpha_{k,0}=\tau$, then we can obtain
%\begin{equation*}
%        \begin{aligned}
%            f(x_{k+1}) - f(x_k)
%            \leq -\frac{\gamma}{6}{\alpha}_{k}^3\|\overline{d}_k\|^3 
%            =-\frac{\gamma}{6}\tau^3\|\overline{d}_k\|^3
%            \leq -\frac{\sqrt{3}\gamma\tau^3}{16}
%        \end{aligned}
%\end{equation*}
%where the last inequality is due to $\|s_k\| = \sqrt{1-|t_k|^2} \geq \sqrt{1-\nu^2} \geq \sqrt{3}/2$.
%Moreover, if the backtracking line search terminates with $\alpha_k=\alpha_{k,0}=\tau \cdot \alpha_k^{\text{max}}$, then
%\begin{equation*}
%        \begin{aligned}
%            f(x_{k+1}) - f(x_k)
%            \leq -\frac{\gamma}{6}{\alpha}_{k}^3\|\overline{d}_k\|^3 
%            \leq -\frac{\sqrt{3}\gamma}{16}\left(\tau \alpha_k^{\text{max}}\right)^3 
%            \leq -\frac{\sqrt{3}\gamma}{16}(\tau\bar{\alpha})^3,
%        \end{aligned}
%\end{equation*}
%where $\bar{\alpha}>0$ is a constant, and the last inequality follows from assumption (A1) and the boundedness of $\|d_k\|$. 
%Indeed, we have
%\begin{equation*}
%\|d_k\| \leq \gamma_D \|s_k\| \leq \gamma_D.
%\end{equation*}
It follows from assumption (A1) that there exists a positive constant $\underline{\xi}$ such that all iterates $\{x_k\}$ satisfy
\begin{equation*}
\min_{i}\left\{ x_k^i - l^i,\ u^i - x_k^i \right\}\geq \underline{\xi}>0.
\end{equation*}
Furthermore, combining the boundedness of $\|d_k\|$ and the definition of $\alpha_k^{\text{max}}$, we conclude that the initial line search stepsize $\alpha_{k,0}$ has a uniform positive lower bound, i.e., $\alpha_{k,0} \geq \underline{\alpha}>0$ for all $k$.

If the backtracking line search terminates with $\alpha_k=\alpha_{k,0}$, then we can obtain
\begin{equation*}
        \begin{aligned}
            f(x_{k+1}) - f(x_k)
            \leq -\frac{\gamma}{6}{\alpha}_{k}^3\|\overline{d}_k\|^3 
            = -\frac{\gamma}{6}\alpha_{k,0}^3 \|s_k\|^3
            \leq -\frac{\sqrt{3}\gamma}{16}\underline{\alpha}^3,
        \end{aligned}
\end{equation*}
where the last inequality is due to $\|s_k\| = \sqrt{1-|t_k|^2} \geq \sqrt{1-\nu^2} \geq \sqrt{3}/2$.

Otherwise, if the backtracking line search not stop at the iteration $j \geq 0$ and the decrease condition is not met, i.e., $f(x_k + {\alpha}_{k,j}d_k) - f(x_k) > -\frac{\gamma}{6}\alpha_{k,j}^3\|\overline{d}_k\|^3 = -\frac{\gamma}{6}\alpha_{k,j}^3\|s_k\|^3$. Using the \eqref{lemma3.1} in Lemma \ref{lemma3} and Lemma \ref{lemma5}, we have the following results
\begin{equation}
        \label{lemma7.1}
       \overline{d}_k^T \overline{B}_k \overline{d}_k = - \theta_k \|\overline{d}_k\|^2 - t_k^2 \cdot (\delta - \theta_k) \quad \text{and} \quad
        \overline{g}_k^T \overline{d}_k = |t_k| \cdot(\delta - \theta_k).
    \end{equation}
From the definition of $B_k$ in \eqref{fangshedk}
\begin{equation}\label{lemma7.2}
       \overline{d}_k^T \overline{B}_k \overline{d}_k = \overline{d}_k^T \left(\overline{H}_k+\overline{C}_k\right) \overline{d}_k 
        ={d}_k^T H_k {d}_k + \overline{d}_k^T \overline{C}_k \overline{d}_k
\end{equation}
holds. Based on the above analysis, the following result
\begin{equation}\label{lemma7.3}
          {d}_k^T{H}_k {d}_k=- \theta_k \|\overline{d}_k\|^2 - t_k^2 \cdot (\delta - \theta_k)-\overline{d}_k^T \overline{C}_k \overline{d}_k \leq \theta_k \|\overline{d}_k\|^2 - t_k^2 \cdot (\delta - \theta_k)
\end{equation}
can be obtained due to the fact that $\overline{d}_k^T \overline{C}_k \overline{d}_k \geq 0$.
$\alpha_{k,j} \in (0,1)$, so $\alpha_{k,j} - \frac{\alpha_{k,j}^2}{2} \geq 0$ and further
\begin{equation}\label{lemma7.4}
\left(\alpha_{k,j} - \frac{\alpha_{k,j}^2}{2}\right) \cdot (\delta-\theta_k) \leq 0.
\end{equation}
With the results derived above and \eqref{lemma6.1b}, we get that
\begin{eqnarray}\label{lemma7.5}
            &&-\frac{\gamma}{6}\alpha_{k,j}^3\|s_k\|^3 \notag\\
                                                 & <& f(x_k + \alpha_{k,j} d_k) - f(x_k)\notag\\                                                                                                                                                       
                                                 &\overset{\rm\eqref{lemma6.1b}} \leq& \alpha_{k,j} \cdot g_k^T d_k + \frac{\alpha_{k,j}^2}{2} \cdot d_k^T H_k d_k + \frac{L_H}{6}\alpha_{k,j}^3\|d_k\|^3\notag\\ 
                                                 &=& \alpha_{k,j} \cdot \overline{g}_k^T \overline{d}_k + \frac{\alpha_{k,j}^2}{2} \cdot d_k^T H_k d_k + \frac{L_H}{6}\alpha_{k,j}^3\|d_k\|^3\notag\\                                                                                  
                                                 &\overset{\rm\eqref{lemma7.1},\rm\eqref{lemma7.3}} \leq& \alpha_{k,j} \cdot |t_k| \cdot (\delta-\theta_k) - \frac{\alpha_{k,j}^2}{2} \cdot \theta_k\|\overline{d}_k\|^2 - \frac{\alpha_{k,j}^2}{2} \cdot t_k^2 \cdot (\delta - \theta_k)  + \frac{L_H}{6}\alpha_{k,j}^3\gamma_D^3\|\overline{d}_k\|^3    \notag\\
                                                 &\leq &\left(\alpha_{k,j} - \frac{\alpha_{k,j}^2}{2}\right) \cdot t_k^2 \cdot (\delta-\theta_k) -  \frac{\alpha_{k,j}^2}{2} \cdot \theta_k \|s_k\|^2 + \frac{L_H}{6}\alpha_{k,j}^3\gamma_D^3\|s_k\|^3                            \notag\\
                                                 &\overset{\rm\eqref{lemma7.4}} \leq& - \frac{\alpha_{k,j}^2}{2} \cdot \theta_k\|s_k\|^2 + \frac{L_H}{6}\alpha_{k,j}^3\gamma_D^3 \|s_k\|^3 \notag\\                                                                                                         
                                                 & \leq& - \frac{\alpha_{k,j}^2}{2} \cdot \delta \|s_k\|^2 + \frac{L_H}{6}\alpha_{k,j}^3\gamma_D^3\|s_k\|^3.
\end{eqnarray} 
Therefore, $\alpha_{k,j} \geq \frac{3\delta}{\left(L_H\gamma_D^3+\gamma\right)\|s_k\|}$ holds, which further implies that
\begin{equation}\label{j}
 \beta^j\geq \frac{3\delta}{\left(L_H\gamma_D^3+\gamma\right)\|s_k\|\alpha_{k,0}}.
\end{equation}
Then, due to $\|s_k\| \leq 1$ and $\alpha_{k,0}\in(0,1)$
\begin{equation*}
j_N \geq \log_\beta\left(\frac{3\delta}{\left(L_H\gamma_D^3+\gamma\right)\|s_k\|\alpha_{k,0}}\right), 
\end{equation*}
where $j_N := \left \lceil \log_\beta\left(\frac{3\delta}{L_H\gamma_D^3+\gamma}\right) \right\rceil$. Such $j=j_N$ does not satisfy the inequality \eqref{lemma7.5}, which means that decrease condition holds. Therefore, the iteration number of backtracking subroutine $j_k$ is upper bounded by $j_N$.
%Assume that $\alpha_{\text{last}}$ is the last stepsize that fails to meet \eqref{decrease condition}. Thus, it follows that 
%\begin{equation*}
%\alpha_{\text{last}} \geq \frac{3\delta}{\left(L_H\gamma_D^3+\gamma\right)\|v_k\|},
%\end{equation*}
%which implies that $\alpha_k$ is the first stepsize to satisfy the condition \eqref{decrease condition}
%\begin{equation*}
%\alpha_k = \beta \cdot \alpha_{\text{last}} \geq \beta \cdot \frac{3\delta}{\left(L_H\gamma_D^3+\gamma\right)\|v_k\|}.
%\end{equation*}
The function value decreases as follows
\begin{equation*}
        \begin{aligned}
            &f(x_{k+1}) - f(x_k)\\
             & \leq -\frac{\gamma}{6}\alpha_k^3\|s_k\|^3 
             \leq -\frac{\gamma\beta^3}{6}\alpha_{k,0}^3 \beta^{3(j-1)}\|s_k\|^3 
              \overset{\rm{\eqref{j}}}\leq -\frac{\gamma}{6} \cdot \frac{27\delta^3 \beta^3 }{\left(L_H\gamma_D^3+\gamma\right)^3}
             = - \frac{9\gamma\beta^3\delta^3}{2\left(L_H\gamma_D^3+\gamma\right)^3}.
        \end{aligned}
\end{equation*}
%where the last inequality comes from $\beta^{j_k - 1} \geq \frac{3\delta}{(L_H\gamma_D^3+\gamma)\|v_k\|}$.
%Since $\{x_k\}$ is the sequence generated by Algorithm 1, $\{f(x_k)\}$ is monotonically decreasing. Together with ${\alpha}_{k,j}=\min\left\{ \hat{\alpha}_{k,j}, \alpha_k^{\text{max}} \right\} \leq \hat{\alpha}_{k,j}$, we obtain that
%\begin{equation*}
%f(x_{k+1}) - f(x_k)= f(x_k + {\alpha}_{k,j}d_k) - f(x_k) \leq  f(x_k + \hat{\alpha}_{k,j}d_k) - f(x_k).
%\end{equation*}

Based on the function value decrease amounts corresponding to the two cases, we have
\begin{equation*}
\begin{aligned}
f(x_{k+1}) - f(x_k) &\leq -\min \left\{\frac{\sqrt{3}\gamma}{16} \underline{\alpha}^3, \frac{9\gamma\beta^3\delta^3}{2(L_H\gamma_D^3+\gamma)^3}\right\}\\
%&= \min \left\{-\frac{\sqrt{3}\gamma}{16}, -\frac{9\gamma\beta^3\delta^3}{2\left(L_H\gamma_D^3+\gamma\right)^3}\right\},
\end{aligned}
\end{equation*}
%where the last equation holds because $\alpha_k^{\text{max}} \leq 1$. This completes the proof.
\end{proof}

\begin{lemma}\label{lemma8}
    Suppose that assumptions (A1)-(A3) hold and set $\nu \in (0, 1/2)$. If $|t_k| \ge \nu$ and $\|s_k / t_k\| > \Delta$, then let $\overline{d}_k = s_k/t_k$. The backtracking line search terminates with $\alpha_{k,j} = \beta^{j_k} \alpha_{k,0}$, and $j_k$ is upper bounded by
\begin{equation*}
        j_N := \left \lceil \log_\beta\left(\frac{3\delta\nu}{L_H\gamma_D^3+\gamma}\right) \right\rceil,
\end{equation*}
and the function value associated with the stepsize $\alpha_k$ satisfies
    \begin{equation}\label{decrease2}
        f(x_{k+1}) - f(x_k) \leq -\min \left\{-\frac{\gamma \underline{\alpha}^3}{6}\Delta^3, \frac{9\gamma\beta^3\delta^3}{2\left(L_H\gamma_D^3+\gamma\right)^3}\right\}.
    \end{equation}
\end{lemma}
\begin{proof}
%Similar to Lemma \ref{lemma7}, suppose that the backtracking line search terminates with $\alpha_k=\alpha_{k,0}=1$, we have
%\begin{equation*}
%        \begin{aligned}
%            f(x_{k+1}) - f(x_k)
%             \leq -\frac{\gamma}{6}\alpha_k^3\|\overline{d}_k\|^3  
%             = -\frac{\gamma}{6}\alpha_k^3\Delta^3
%             \leq -\frac{\gamma\Delta^3}{6},
%        \end{aligned}
%\end{equation*}
%where the last inequality comes from $\|\overline{d}_k\| > \Delta$. When $\alpha_k=\alpha_{k,0}=\alpha_k^{\text{max}}$, it follows that
%\begin{equation*}
%        \begin{aligned}
%            f(x_{k+1}) - f(x_k)
%             \leq -\frac{\gamma}{6}\alpha_k^3\|\overline{d}_k\|^3  
%             \leq -\frac{\gamma\Delta^3}{6} \left(\alpha_k^{\text{max}}\right)^3
%             \leq -\frac{\gamma\bar{\alpha}^3}{6} \Delta^3.
%        \end{aligned}
%\end{equation*}\\
%The last inequality follows by a similar argument as in Lemma \ref{lemma7}. 
%Here $\|d_k\|$ is also bounded as follows
%\begin{equation}\label{dkbodcase2}
%\|d_k\| \leq \gamma_D \cdot \frac{\|s_k\|}{|t_k|}=\gamma_D \cdot \frac{\sqrt{1-|t_k|^2}}{|t_k|} \leq \frac{\gamma_D}{\nu}.
%\end{equation}
Similar to Lemma \ref{lemma7}, suppose that the backtracking line search terminates with $\alpha_k=\alpha_{k,0}$, we have
\begin{equation*}
        \begin{aligned}
            f(x_{k+1}) - f(x_k)
             \leq -\frac{\gamma}{6}\alpha_k^3\|\overline{d}_k\|^3  
             \leq -\frac{\gamma}{6}\alpha_{k,0}^3\Delta^3
             \leq -\frac{\gamma \underline{\alpha}^3}{6}\Delta^3,
        \end{aligned}
\end{equation*}

If $\alpha_k =\alpha_{k,0}$ does not lead to a sufficient decrease, then for any $j \geq 0$ where the decrease condition is not met. When $\overline{d}_k = s_k / t_k$, using the \eqref{lemma3.3} in Lemma \ref{lemma3} and Lemma \ref{lemma5}, we have the following results
\begin{equation}
        \label{lemma8.1}
       \overline{d}_k^T \overline{B}_k \overline{d}_k = - \theta_k \|\overline{d}_k\|^2 - \overline{g}_k^T \overline{d}_k \quad \text{and} \quad
        \overline{g}_k^T \overline{d}_k = \left(\delta - \theta_k\right)\leq 0.
    \end{equation} 
Combine \eqref{lemma7.2} and \eqref{lemma8.1}, we have that 
\begin{equation}\label{lemma8.2}
          {d}_k^T {H}_k {d}_k=- \theta_k \|\overline{d}_k\|^2 - \left(\delta - \theta_k\right)-\overline{d}_k^T \overline{C}_k \overline{d}_k \leq - \theta_k \|\overline{d}_k\|^2 - \left(\delta - \theta_k\right).
\end{equation}
In combination with the results obtained above, we obtain
\begin{eqnarray}\label{lemma8.3}
           && -\frac{\gamma}{6}\alpha_{k,j}^3\|\overline{d}_k\|^3 \notag\\
                                                 & <& f(x_k + \alpha_{k,j} d_k) - f(x_k) \notag\\                                                                                                      
                                                 &\overset{\rm{\eqref{lemma6.1b}}} \leq& \alpha_{k,j} \cdot \overline{g}_k^T \overline{d}_k + \frac{\alpha_{k,j}^2}{2} \cdot d_k^T H_k d_k + \frac{L_H}{6}\alpha_{k,j}^3\|d_k\|^3 \notag\\ 
                                                 &\overset{\rm{\eqref{lemma8.2}}} \leq &\alpha_{k,j} \cdot \left(\delta - \theta_k\right) - \frac{\alpha_{k,j}^2}{2} \cdot \theta_k\|\overline{d}_k\|^2 - \frac{\alpha_{k,j}^2}{2}\cdot \left(\delta - \theta_k\right)  + \frac{L_H}{6}\alpha_{k,j}^3\gamma_D^3\|\overline{d}_k\|^3 \notag\\ 
                                                 & =&\left(\alpha_{k,j} - \frac{\alpha_{k,j}^3}{2}\right) \cdot \left(\delta - \theta_k\right) -\frac{\alpha_{k,j}^2}{2}\theta_k\|\overline{d}_k\|^2 + \frac{L_H}{6}\alpha_{k,j}^3\gamma_D^3\|\overline{d}_k\|^3 \notag\\ 
                                                 &\overset{\rm\eqref{lemma7.4}} \leq& -\frac{\alpha_{k,j}^2}{2}\delta\|\overline{d}_k\|^2 + \frac{L_H}{6}\alpha_{k,j}^3\gamma_D^3\|\overline{d}_k\|^3.                    
\end{eqnarray} 
Therefore, $\alpha_{k,j}\geq \frac{3\delta}{\left(L_H\gamma_D^3+\gamma\right)\|\overline{d}_k\|}$ holds, which further implies that
\begin{equation}\label{j2}
\beta^j\geq \frac{3\delta}{\left(L_H\gamma_D^3+\gamma\right)\|\overline{d}_k\|\alpha_{k,0}}.
\end{equation}
Note that
\begin{equation*}
\|\overline{d}_k\|=\frac{\|s_k\|}{|t_k|}=\frac{\sqrt{1-|t_k|^2}}{|t_k|} \leq \frac{1}{\nu}
\end{equation*}
holds due to $|t_k| \geq \nu$. Thus,
\begin{equation*}
j_N \geq \log_\beta\left(\frac{3\delta}{\left(L_H\gamma_D^3+\gamma\right)\|\overline{d}_k\|\alpha_{k,0}}\right), 
\end{equation*}
where $j_N := \left \lceil \log_\beta\left(\frac{3\delta\nu}{L_H\gamma_D^3+\gamma}\right) \right\rceil$. This means that inequality \eqref{lemma8.3} does not hold when $j=j_N$, which further shows that decrease condition is satisfied. It follows that the iteration number of backtracking subroutine $j_k$ is upper bounded by $j_N$,
%Similar to Lemma \ref{lemma7}, 
and the function value decreases as follows
\begin{equation*}
         \begin{aligned}
            &f(x_{k+1}) - f(x_k)\\
              & \leq -\frac{\gamma}{6}\alpha_k^3\|\overline{d}_k\|^3 
             \leq -\frac{\gamma\beta^3}{6}\alpha_{k,0}^3 \beta^{3(j-1)}\|\overline{d}_k\|^3 
              \overset{\rm{\eqref{j2}}}\leq -\frac{\gamma}{6} \cdot \frac{27\delta^3 \beta^3 }{\left(L_H\gamma_D^3+\gamma\right)^3}
             = - \frac{9\gamma\beta^3\delta^3}{2\left(L_H\gamma_D^3+\gamma\right)^3}.
        \end{aligned}
\end{equation*}
In summary, the function value associated with the stepsize $\alpha_k$ satisfies 
\begin{equation*}
\begin{aligned}
        f(x_{k+1}) - f(x_k) &\leq - \min \left\{\frac{\gamma\underline{\alpha}^3}{6} \Delta^3, \frac{9\gamma\beta^3\delta^3}{2\left(L_H\gamma_D^3+\gamma\right)^3}\right\}\\
      %  &\leq \min \left\{\frac{\gamma\Delta^3}{6}, \frac{9\gamma\beta^3\delta^3}{2\left(L_H\gamma_D^3+\gamma\right)^3}\right\}.
        \end{aligned}
\end{equation*}
Then, the proof is complete.
\end{proof}

By combining the two lemmas mentioned above, we can now establish a unified descent property for homogenised negative curvature with a backtracking line search.
\begin{lemma}\label{lemma9}
Suppose that assumptions (A1)-(A3) hold and set $\nu \in (0, 1/2)$. Let the backtracking line search parameters $\beta, \gamma$ satisfy $\beta \in (0, 1)$ and $\gamma > 0$. Then, after every outer iterate, the function value decreases as
\begin{equation*}
        f(x_{k+1}) - f(x_k) \leq -\min\left\{\frac{\sqrt{3}\gamma}{16} \underline{\alpha}^3, \frac{\gamma\underline{\alpha}^3}{6} \Delta^3, \frac{9\gamma\beta^3\delta^3}{2\left(L_H\gamma_D^3+\gamma\right)^3}\right\}.
\end{equation*}
    and the iteration for backtracking line search is at most
\begin{equation*}
        j_N \leq \max\left\{\left\lceil \log_\beta\left(\frac{3\delta}{L_H\gamma_D^3+\gamma}\right) \right\rceil, \left \lceil \log_\beta\left(\frac{3\delta\nu}{L_H\gamma_D^3+\gamma}\right) \right\rceil\right\} = \left \lceil \log_\beta\left(\frac{3\delta\nu}{L_H\gamma_D^3+\gamma}\right) \right\rceil.
\end{equation*}
\end{lemma}

\subsection{Analysis for the scenario of small value}\label{section4.2}
In this subsection, we consider the small value case where $\|\overline{d}_k\| \le \Delta$. Note that in the case of $|t_k| \ge \sqrt{1/(1+\Delta^2)}$, we set the homogenised direction as $\overline{d}_k=s_k/t_k$. Evidently, $\|\overline{d}_k\| = \|s_k\|/|t_k| = \sqrt{1-|t_k|^2}/|t_k| \le \Delta$ holds, validating the name of the small value case in Algorithm \ref{SOBASIP}. Under this scenario, 
we will proceed to prove that the next iteration point $x_{k+1}=x_k+D_k^{-1}\overline{d}_k$ is a $\epsilon$-approximate SOSP. Consequently, for the small value case, terminating the algorithm after one iteration is feasible. To prove this result, we provide an upper bound of $\|g_k\|$ for preparation.
\begin{lemma}\label{lemma10}
Suppose that assumptions (A2)-(A3) hold. If $g_k \neq 0$, and $\|\overline{d}_k\| \leq \Delta \leq \sqrt{2}/2\gamma_D$, then we have
    \begin{equation}\label{bound.gk}
        \|g_k\| \leq 2U_H \gamma_D \delta\Delta.
    \end{equation}
\end{lemma}
\begin{proof}
By Lemma \ref{lemma2} we see that $\theta_k-\delta >0$. Moreover, together with \eqref{lemma3.3} in Lemma \ref{lemma3}, we can get a upper bound of $\theta_k-\delta$ as follows
\begin{equation}\label{lemma10.1}
\theta_k-\delta \overset{\rm{\eqref{lemma3.3}}} =- \overline{g}_k^T \overline{d}_k =-g_k^T d_k \leq \|g_k\|\|D_k^{-1}\overline{d}_k\| \overset{\rm{\eqref{boundB.D}}} \leq \gamma_D \Delta\|g_k\|
\end{equation}
Define a univariate function $h(t) = t^2 + \delta(g_k^T H_k g_k / \|g_k\|^2)t  - \|g_k\|^2$. Moreover, it can be readily concluded that the equation $h(t)=0$ has two real roots of opposite signs. We let the positive real root be $t_2$. Employing the above inequality $\theta_k-\delta > 0$, we have $\theta_k-\delta \geq t_2$. Furthermore, it follows that 
\begin{equation*}
t_2 \leq \theta_k-\delta \overset{\rm{\eqref{lemma10.1}}} \leq \gamma_D \Delta\|g_k\|,
\end{equation*}
which implies that 
\begin{equation*}
        h\left(\gamma_D \Delta\|g_k\|\right) =
        \gamma_D^2\Delta^2\|g_k\|^2+ \frac{g_k^T H_k g_k}{\|g_k\|^2} \gamma_D \delta \Delta\|g_k\|-\|g_k\|^2\geq 0.
\end{equation*}
 After some algebra, we obtain
\begin{eqnarray*}
        \|g_k\|
         \leq \frac{g_k^T H_k g_k / \|g_k\|^2 \gamma_D \delta \Delta}{1-\gamma_D^2\Delta^2} 
         \overset{\rm{\eqref{bound.Hk}}} \leq \frac{U_H \gamma_D \delta\Delta}{1-\gamma_D^2\Delta^2}
         \leq 2 U_H \gamma_D \delta \Delta.
\end{eqnarray*}
The last inequality follows from $\Delta \leq \sqrt{2}/2\gamma_D$.
\end{proof}

The following lemma shows that $\|\overline{g}_{k+1}\|$ has an upper bound, while the smallest eigenvalue of the $\overline{B}_{k+1}$ has a lower bound.
\begin{lemma}\label{lemma11}
    Suppose that assumptions (A1)-(A3) hold. If $g_k \neq 0$, and $\| \overline{d}_k \| \leq \Delta$, then let $\alpha_k = 1$, we have
\begin{equation}\label{gk+1.bounded}
        \|\overline{g}_{k+1}\| \leq \frac{L_H^2\gamma_D}{4}\Delta^2 + 2 U_H \gamma_D^2 \delta \Delta, 
\end{equation}
and
\begin{equation}\label{Bk+1.bounded}
        \overline{B}_{k+1} \succeq - \left(L_H\gamma_D^2 + U_H + 4U_H \delta\right)\gamma_D \Delta I - \left( \frac{L_H^2}{2}U_H +\frac{L_H}{2}\gamma_D +2U_H\delta\gamma_D\right)\gamma_D\Delta^2 I-\delta I.
\end{equation}
\end{lemma}
\begin{proof}
We first prove \eqref{gk+1.bounded}. With the uniform continuity of $\{g_k\}$, we can bound the norm of $\|g_{k+1}\|$ as follows
\begin{equation*}
          \|g_{k+1}\| \leq \|g_{k+1}-g_k\| + \|g_k\|
                      %&\overset{\rm{\eqref{lemma6.1a}}}\leq \|H_kd_k\| + \frac{M}{2}\|d_k\|^2+ \|g_k\|\\
%                      &\overset{\rm{\eqref{bound.Hk},\eqref{boundB.D}}}\leq  \gamma_DU_H\|\overline{d}_k\| + \frac{M}{2}\gamma_D^2\|\overline{d}_k\|^2+ \|g_k\|\\
                      %&\overset{\rm{\eqref{bound.gk}}}\leq  \gamma_DU_H \Delta + \frac{M}{2}\gamma_D^2\Delta^2 + 2 U_H \gamma_D \delta \Delta\\
                      \overset{\rm{\eqref{bound.gk}}} \leq \frac{L_H^2}{4}\Delta^2 + 2 U_H \gamma_D \delta \Delta.
                      %& = \frac{M}{2}\gamma_D^2\Delta^2 + (2\gamma_D\delta+ 3 \gamma_DU_H)\Delta.
\end{equation*}
Further
\begin{equation*}
%\|\overline{g}_{k+1}\| \leq \frac{M}{2}\Delta^2\gamma_D^3 +  (2\delta\gamma_D^2+3 U_H\gamma_D^2)\Delta.
\|\overline{g}_{k+1}\| \leq \gamma_D\|g_{k+1}\| \leq \frac{L_H^2\gamma_D}{4}\Delta^2 + 2 U_H \gamma_D^2 \delta \Delta.
\end{equation*}
Next, we prove \eqref{Bk+1.bounded}. Lemma \ref{lemma1} and optimality condition \eqref{lemma2.1} in Lemma \ref{lemma2} imply that
\begin{equation*}
\overline{B}_k + \theta_k \cdot I \succeq 0.
\end{equation*}
With \eqref{bound.gk} and \eqref{lemma10.1}, we further obtain 
\begin{align}
\overline{B}_k &\overset{\rm{\eqref{lemma10.1}}} \succeq -\theta_kI \succeq -\left(\gamma_D\Delta\|g_k\|+\delta\right)I \notag\\
    &\overset{\rm{\eqref{bound.gk}}} \succeq -2U_H\delta\gamma_D^2\Delta^2I-\delta I.
\label{lemma11.2}
\end{align}
Before bounding the $\overline{B}_{k+1}$, we provide an upper bound of the $\|\overline{B}_{k+1}-\overline{B}_k\|$ for preparation
\begin{eqnarray}\label{lemma11.3}
&&\|\overline{B}_{k+1}-\overline{B}_k\| \notag\\
&\leq& \|D_{k+1}^{-1}H_{k+1}D_{k+1}^{-1} -D_k^{-1}H_kD_k^{-1}\| +\|\operatorname{diag}\{g_{k+1}\}J_{k+1}^v-\operatorname{diag}\{g_k\}J_k^v\|. 
\end{eqnarray}
For the first half of \eqref{lemma11.3}, by employing the fact that $\{D_k^{-1}\}$ is uniformly continuous and $H(x)$ is Lipschitz continuous on $X$, we get 
%\begin{align}
%&&\|D_{k+1}^{-1}H_{k+1}D_{k+1}^{-1} -D_k^{-1}H_kD_k^{-1}\| \notag\\
%& =& \|D_{k+1}^{-1}(H_{k+1}-H_k)D_{k+1}^{-1}+(D_{k+1}^{-1}-D_k^{-1})H_kD_{k+1}^{-1} + D_k^{-1}H_k(D_{k+1}^{-1}-D_k^{-1})\| \notag\\
%& \leq &\|D_{k+1}^{-1}\|^2\|H_{k+1}-H_k\|+\|D_{k+1}^{-1}-D_k^{-1}\| \|H_k\| \|D_{k+1}^{-1}\|+ \|D_{k+1}^{-1}-D_k^{-1}\| \|H_k\| \|D_k^{-1}\| \notag\\
%& \overset{\rm{\eqref{lemma6.1b}}}\leq& L_H\gamma_D^2 \|d_k\| + 2\|D_{k+1}^{-1}-D_k^{-1}\| U_H\gamma_D \notag\\
%& \overset{\rm{\eqref{uniformly continuous}}}\leq& L_H\gamma_D^3 \|\overline{d}_k\| + \frac{L_H^2}{2}U_H \gamma_D\Delta^2 \notag\\
%& \leq& L_H\gamma_D^3 \Delta + \frac{L_H^2}{2}U_H \gamma_D\Delta^2.
%\label{lemma11.4}
%\end{align}
\begin{eqnarray}\label{lemma11.4}
&&\|D_{k+1}^{-1}H_{k+1}D_{k+1}^{-1} - D_k^{-1}H_kD_k^{-1}\| \notag\\
&=& \|D_{k+1}^{-1}(H_{k+1}-H_k)D_{k+1}^{-1} + (D_{k+1}^{-1}-D_k^{-1})H_kD_{k+1}^{-1} + D_k^{-1}H_k(D_{k+1}^{-1}-D_k^{-1})\| \notag\\
&\leq& \|D_{k+1}^{-1}\|^2\|H_{k+1}-H_k\| + \|D_{k+1}^{-1}-D_k^{-1}\| \|H_k\| (\|D_{k+1}^{-1}\| + \|D_k^{-1}\|) \notag\\
&\overset{\rm{\eqref{H.lipschitz},\eqref{bound.Hk}}}\leq& L_H\gamma_D^2 \|d_k\| + 2\|D_{k+1}^{-1}-D_k^{-1}\| U_H\gamma_D \notag\\
&\overset{\rm{\eqref{uniformly continuous}}}\leq& L_H\gamma_D^3 \|\overline{d}_k\| + \frac{L_H^2}{2}U_H \gamma_D\Delta^2 \notag\\
&\leq& L_H\gamma_D^3 \Delta + \frac{L_H^2}{2}U_H \gamma_D\Delta^2.
\end{eqnarray}
When considering the second half of \eqref{lemma11.3}, we can derive that as follows by making use of the property that $g_k$ is bounded above which is established in Lemma \ref{lemma10}.
\begin{eqnarray}\label{lemma11.5}
&& \|\operatorname{diag}\{g_{k+1}\}J_{k+1}^v-\operatorname{diag}\{g_k\}J_k^v\|\notag\\
& = & \|\operatorname{diag}\{g_{k+1}\}J_{k+1}^v-\operatorname{diag}\{g_k\}J_{k+1}^v+\operatorname{diag}\{g_k\}J_{k+1}^v-\operatorname{diag}\{g_k\}J_k^v\|\notag\\
& \leq& \|\operatorname{diag}\{g_{k+1}-g_k\}J_{k+1}^v\| + \|\operatorname{diag}\{g_k\}(J_{k+1}^v-J_k^v)\|\notag\\
& \leq& \|\operatorname{diag}\{g_{k+1}-g_k\}\| \|J_{k+1}^v\| + \|\operatorname{diag}\{g_k\}\| \|(J_{k+1}^v-J_k^v)\| \notag\\
& \leq& \|g_{k+1}-g_k\| \|J_{k+1}^v\| + \|g_k\| \|(J_{k+1}^v-J_k^v)\|\notag\\
%& \overset{\rm{\eqref{lemma6.1a},\eqref{lemma10.1}}} \leq (\|H_kd_k\| + \frac{M}{2}\|d_k\|^2) \|J_{k+1}^v\| +  2(U_H+\delta) \gamma_D \Delta\|(J_{k+1}^v-J_k^v)\|\\
& \overset{\rm{\eqref{lemma6.1a},\eqref{lemma10.1}}} \leq& \left(\|H_kd_k\| + \frac{L_H}{2}\|d_k\|^2\right) \|J_{k+1}^v\| +  2 U_H \gamma_D \delta \Delta\|(J_{k+1}^v-J_k^v)\|\notag\\
& \leq& \left(U_H\gamma_D\|\overline{d}_k\|+\frac{L_H}{2}\gamma_D^2\|\overline{d}_k\|^2\right) \|J_{k+1}^v\| +  2 U_H \gamma_D \delta \Delta\|(J_{k+1}^v-J_k^v)\|\notag\\
%& \leq U_H\gamma_D\Delta+\frac{M}{2}\gamma_D^2\Delta^2 + 4(U_H+\delta) \gamma_D \Delta.
& \leq& U_H\gamma_D\Delta+\frac{L_H}{2}\gamma_D^2\Delta^2 + 4 U_H \gamma_D \delta \Delta.
\end{eqnarray}
The last inequality follows from $J_k^v$ being a diagonal matrix, where diagonal entries are 1, -1, and 0, this further implies that $\|J_{k+1}^v\| \leq 1$ and $\|(J_{k+1}^v-J_k^v)\| \leq 2$.
Rearranging the terms above, we conclude that 
\begin{eqnarray}\label{lemma11.6}
&& \|\overline{B}_{k+1}-\overline{B}_k\|\notag\\
%&\leq \sqrt{2M\Delta}U_H\gamma_D + (5U_H\gamma_D + 4\delta)\Delta + M\gamma_D^3 \Delta + \frac{M}{2}\gamma_D^2\Delta^2 
& \overset{\rm{\eqref{lemma11.4}.\eqref{lemma11.5}}}\leq& \left(L_H\gamma_D^3 + U_H\gamma_D + 4U_H\gamma_D \delta\right) \Delta + \left( \frac{L_H^2}{2}U_H \gamma_D +\frac{L_H}{2}\gamma_D^2 \right)\Delta^2.
\end{eqnarray}
Finally, together with \eqref{lemma11.2} and \eqref{lemma11.6}, we are able to establish a lower bound on $\overline{B}_{k+1}$
\begin{equation*}
\begin{aligned}
\overline{B}_{k+1} &\succeq \overline{B}_k - \|\overline{B}_{k+1}-\overline{B}_k\|I\\
%        &\succeq -\sqrt{2M\Delta}U_H\gamma_D I- (5U_H\gamma_D + 4\delta+M\gamma_D^3)\Delta I-2(U_H+\delta+\frac{M}{2})\gamma_D^2\Delta^2I-\delta I
& \geq - \left(L_H\gamma_D^2 + U_H + 4U_H \delta\right) \gamma_D\Delta I - \left( \frac{L_H^2}{2}U_H  +\frac{L_H}{2}\gamma_D +2U_H\delta\gamma_D\right)\gamma_D\Delta^2 I-\delta I.
\end{aligned}
\end{equation*}
The proof is then complete.
\end{proof}

\subsection{The global convergence}\label{4.3}
Based on the previous discussions, we will present the global convergence results of SOBASIP in Theorem \ref{theorem1}. It shows that with appropriate selections of the perturbation parameter $\delta$ and $\Delta$, our algorithm can attain an iteration complexity of $O(\epsilon^{-3/2})$ when seeking an $\epsilon$-approximate second-order stationary point.
\begin{theorem}\label{theorem1}
Suppose that assumptions (A1)-(A3) hold, Let $\delta = \sqrt{\epsilon}$, $\Delta = 2\sqrt{\epsilon} / L_H$ and $\nu \in (0, 1/2)$, and the backtracking line search parameters $\beta, \gamma$ satisfy $\beta \in (0, 1)$ and $\gamma > 0$. Then Algorithm \ref{SOBASIP} terminates in at most $O\left(\epsilon^{-3/2}\log_\beta(\epsilon) \right)$ steps, and the next iterate $x_{k+1}$ is a SOSP. Specifically, the number of iterations is bounded by
\begin{equation*}
        O\left(\max\left\{\frac{3L_H^3}{4\gamma\underline{\alpha}^3}, \frac{2(L_H\gamma_D^3+\gamma)^3}{9\gamma\beta^3}\right\}\left \lceil \log_\beta\left(\frac{3\sqrt{\epsilon}\nu}{L_H\gamma_D^3+\gamma}\right) \right\rceil\left(f(x_1) - f_{\inf}\right)\epsilon^{-3/2}\right).
    \end{equation*}
\end{theorem}
\begin{proof}
Since we take $\delta = \sqrt{\epsilon}$ and $\Delta = 2\sqrt{\epsilon} / L_H $, by Lemma \ref{lemma9}, it can be readily established that the decrement of the function value is at least $\Omega(\epsilon^{3/2})$ for the large step case, i.e.,
\begin{equation*}
\begin{aligned}
           f(x_{k+1}) - f(x_k) & \leq -\min\left\{\frac{\sqrt{3}\gamma}{16} \underline{\alpha}^3, \frac{\gamma\underline{\alpha}^3}{6} \Delta^3, \frac{9\gamma\beta^3\delta^3}{2\left(L_H\gamma_D^3+\gamma\right)^3}\right\}\\
%                               &= -\min\left\{\frac{\sqrt{3}\gamma}{16}, \frac{4\gamma\epsilon^{3/2}}{3M^3}, \frac{9\gamma\beta^3\epsilon^{3/2}}{2(M\gamma_D^3+\gamma)^3}\right\} \\   
                                &\leq -\min\left\{\frac{4\gamma\underline{\alpha}^3}{3L_H^3}, \frac{9\gamma\beta^3}{2(L_H\gamma_D^3+\gamma)^3}\right\}\epsilon^{3/2}, \\ 
\end{aligned}
\end{equation*}
and the iteration for backtracking line search is at most
\begin{equation*}
\begin{aligned}
j_N \leq \left \lceil \log_\beta\left(\frac{3\delta\nu}{L_H\gamma_D^3+\gamma}\right) \right\rceil
    & = \left \lceil \log_\beta\left(\frac{3\sqrt{\epsilon}\nu}{L_H\gamma_D^3+\gamma}\right) \right\rceil.
\end{aligned}
\end{equation*}
When the algorithm terminates, by means of Lemma \ref{lemma11} we have
\begin{align}
\|\overline{g}_{k+1}\| &\leq \frac{L_H^2\gamma_D}{4}\Delta^2 + 2 U_H \gamma_D^2 \delta \Delta \notag\\
                      &= \frac{4U_H\gamma_D^2\epsilon}{L_H} + \gamma_D \epsilon \notag\\
                      &= O(\epsilon)
\label{theorem1.1}                     
\end{align}
and
\begin{align}
&\lambda_1 (\overline{B}_{k+1}) \notag \\
& \geq - \left(L_H\gamma_D^3 + U_H\gamma_D + 4U_H\gamma_D \delta\right) \Delta  - \left( \frac{L_H^2}{2}U_H \gamma_D +\frac{L_H}{2}\gamma_D^2 +2U_H\delta\gamma_D^2\right)\Delta^2 -\delta \notag \\
& = -\frac{8U_H\gamma_D^2\epsilon^{3/2}}{L_H^2} - \left(\frac{8U_H\gamma_D+2\gamma_D^2}{L_H} + 2U_H\gamma_D\right)\epsilon - \left(\frac{2U_H\gamma_D}{L_H}+ 2\gamma_D^3 + 1\right)\sqrt{\epsilon}  \notag \\
&= \Omega(-\sqrt{\epsilon}).
\label{theorem1.2}
\end{align}

It follows \eqref{theorem1.1} and \eqref{theorem1.2} that the next iterate $x_{k+1}$ satisfied the Definition \ref{definition2} is already a SOSP. It should be noted that the total decreasing amount if the objective function value cannot exceed $f(x)-f_{\inf}$. So, the number of iterations for large step cases is upper bounded by 
\begin{equation*}
        O\left(\max\left\{\frac{3L_H^3}{4\gamma\underline{\alpha}^3}, \frac{2(L_H\gamma_D^3+\gamma)^3}{9\gamma\beta^3}\right\}\left \lceil \log_\beta\left(\frac{3\sqrt{\epsilon}\nu}{L_H\gamma_D^3+\gamma}\right) \right\rceil\left(f(x_1) - f_{\inf}\right)\epsilon^{-3/2}\right).
    \end{equation*}
Meanwhile, it is also the iterative complexity of our algorithm. Since $\beta<1$, this complete the proof.
\end{proof}

\section{The local convergence}\label{section5}
In this section, the local convergence analysis of SOBASIP will be provided. In particular, when $x_k$ is sufficiently close to a SOSP $x_*$, we will show that $\|\overline{d}_k\|\leq \Delta$ is always holds. This conclusion implies that the stepsize $\alpha_k$
is always equal 1, and the line search process is not required. Consequently, the algorithm SOBASIP attains a local superlinear convergence rate through the setting of perturbation parameter $\delta=0$ for the subsequent iterations.

We first make some standard assumptions \cite{Nesterov4} to enhance the local convergence analysis.\\
\noindent \(\mathbf{(A4)}\) \quad Algorithm \ref{SOBASIP} converges to a strict local optimum $x_*$ satisfying that $ \overline{g}(x_*)=0$ and $\overline{B}(x_*) \succ 0$.

From (A4), we immediately realise that there is a small neighborhood for some $r>0$ and $\mu>0$ such that
\begin{equation}\label{R}
\overline{B}(x)\succeq \mu \cdot I \quad \text{for any} \quad  x\in B(x_*,r).
\end{equation}
To prove local superlinear convergence rate, we assume here that $\{D^{-1}(x)\}$ is $L_{D}$-Lipschitz continuous.\\
\noindent \(\mathbf{(A5)}\) \quad $\{D^{-1}(x)\}$ is Lipschitz continuous on $B(x_*,r)$, i.e., there exists a constant $L_D>0$ such that
\begin{equation}\label{D.lipschitz}
        \|D^{-1}(x) - D^{-1}(y)\| \le L_D \|x-y\|, ~\forall x, y \in B(x_*,r),
\end{equation}
\begin{lemma}\label{lemma12}
Suppose that assumption (A4) holds, then $t_k\neq 0$ for sufficiently large $k$.
\end{lemma}
\begin{proof}
We prove it by contradiction. Suppose that $t_k=0$. Then lemma \ref{lemma3} shows that $(-\theta_k, s_k)$ is the eigenpair of the matrix $\overline{B}_k$, implying that
\begin{equation*}
\lambda_1(\overline{B}_k) \leq -\theta_k.
\end{equation*}
By recalling the lemma \ref{lemma2}, we have $\theta_k>0$, and further $\lambda_1(\overline{B}_k)<0$. This conclusion contradicts $\overline{B}_k \succ 0$. The proof is complete.
\end{proof}
The following lemma shows that $\overline{d}_k$ generated by SOBASIP will eventually reduces to the small value case for sufficiently large $k$. Therefore, we choose $\alpha_k$ as 1 and further update to obtain the iteration point $x_{k+1}=x_k+D_k^{-1}\overline{d}_k$.
\begin{lemma}\label{lemma13}
Suppose that assumption (A4) holds. For sufficiently large $k$, we obtain $\|\overline{d}_k\| \leq \Delta$.
\end{lemma}
\begin{proof}
It is clear that $x_k$ arrives at the neighborhood of $x_*$ for sufficiently large $k$, so both $\overline{B}_k$ and $\overline{B}_k+\theta_kI$ are nonsingular. Then, combining the conclusion in lemma \ref{lemma12} and the equation \eqref{lemma3.3} in lemma \ref{lemma3}, we have
\begin{equation*}
\overline{d}_k=-\left(\overline{B}_k+\theta_kI\right)^{-1}\overline{g}_k,
\end{equation*}
and further
\begin{equation}\label{lemma13.1}
\begin{aligned}
\|\overline{d}_k\| &\leq \|(\overline{B}_k+\theta_k I)^{-1}\| \|\overline{g}_k\| \\
& \leq \frac{\|\overline{g}_k\|}{\mu+\theta_k} \leq \frac{\|\overline{g}_k\|}{\mu}.
\end{aligned}
\end{equation}
The second inequality holds due to $\overline{B}_k \succeq \mu I $ and the last inequality follows from $\theta_k >0$. Otherwise, (A4) implies that 
\begin{equation*}
\|\overline{g}_k\| \to 0 \quad \text{as} \quad k \to \infty.
\end{equation*}
In other words, there exists a sufficiently large $K\geq0$, such that
\begin{equation}\label{lemma13.2}
\|\overline{g}_k\| \leq \Delta \mu, \forall k\geq K.
\end{equation}
By making use of \eqref{lemma13.1} and \eqref{lemma13.2}, the desired conclusion $\|\overline{d}_k\| \leq \Delta$ holds.
\end{proof}
In the local phase, we set the perturbation parameter $\delta=0$. Furthermore, the subproblem \eqref{hsodm.subproblem} that we need to solve transforms into 
\begin{equation}
    \label{hsodm.subproblem.loc}
    \begin{aligned}
        \min_{\|[s; t]\| \le 1}
        \psi_k(s, t; 0) := ~ \begin{bmatrix}
        s\\ t
    \end{bmatrix}^T
    \begin{bmatrix}
        \overline{B}_k   & \overline{g}_k     \\
        \overline{g}_k^T & 0
    \end{bmatrix}
    \begin{bmatrix}
        s \\ t
    \end{bmatrix},~ s\in \mathbb{R}^n, t \in \mathbb{R},
    \end{aligned}
\end{equation}
Here, we also denote $[s_k;t_k]$ as the optimal solution to \eqref{hsodm.subproblem.loc}. Combine with the above results, we prove that SOBASIP achieves a local superlinear convergence rate in the following theorem.
\begin{theorem}
    \label{theorem2}
    Suppose that assumptions (A1)-(A5) hold. For sufficiently large $k$, the Algorithm \ref{SOBASIP} is superlinearly convergent to $x_*$, that is,
     $$
        \|{x_{k+1}-x_*} \| \le O\left(\|x_k - x_*\|^{2}\right).
     $$
\end{theorem}
\begin{proof}
By lemma \ref{lemma12}, we have $t_k=0$. Because we set $\delta=0$, we have the following equations from \eqref{lemma3.3}
\begin{equation}\label{theorem2.1}
 \overline{g}_k^T \overline{d}_k = -\theta_k \quad \text{and} \quad \left(\overline{B}_k + \theta_k I\right) \overline{d}_k = - \overline{g}_k,
\end{equation}
which implying that 
%\begin{equation}\label{theorem2.2}
%    \begin{aligned}
%        \|D_k^{-1}\overline{B}_k^{-1} \overline{g}_k + D_k^{-1}\overline{d}_k \| &\overset{\rm{\eqref{theorem2.1}}} =  \| - \theta_k  D_k^{-1} \overline{B}_k^{-1} \overline{d}_k \|\\
%                                & \leq  \|B_k^{-1}  \| \cdot | \theta_k | \|  D_k^{-1} \| \| \overline{d}_k \| \\
%                                & \overset{\rm{\eqref{theorem2.1}}}\leq \frac{\gamma_D}{\mu} \| \overline{g}_k\| \| \overline{d}_k \|^2.\\
%                                & \overset{\rm{\eqref{lemma13.2}}} \leq \Delta \gamma_D  \| \overline{d}_k \|^2.
%    \end{aligned}
%\end{equation}
\begin{eqnarray}\label{theorem2.2}
\|D_k^{-1}\overline{B}_k^{-1} \overline{g}_k + D_k^{-1}\overline{d}_k \| 
&\overset{\rm{\eqref{theorem2.1}}} =&  \| - \theta_k  D_k^{-1} \overline{B}_k^{-1} \overline{d}_k \|  \notag\\
&\leq&  \|B_k^{-1}  \| \cdot | \theta_k | \|  D_k^{-1} \| \| \overline{d}_k \| \notag \\
&\overset{\rm{\eqref{theorem2.1},\eqref{R}}}\leq& \frac{\gamma_D}{\mu} \| \overline{g}_k\| \| \overline{d}_k \|^2 \notag \\
&\overset{\rm{\eqref{lemma13.2}}} \leq& \Delta \gamma_D  \| \overline{d}_k \|^2.
\end{eqnarray}
In what follows, we prove that
\begin{equation}\label{theorem2.3}
\|x_{k} - D_k^{-1}\overline{B}_k^{-1} \overline{g}_k - x_*\| = O\left(\|x_k - x_*\|^{2}\right).
\end{equation}
Define $\Phi(x)=x-D^{-1}(x)\overline{B}^{-1}(x) \overline{g}(x)$. Obviously, $\Phi(x_*) =x_*$ because of $\overline{g}(x_*)=0$. The Taylor expansion of the $\Phi(x_k)$ at $x_*$ is as follows
\begin{equation*}
\Phi(x_k) = \Phi(x_*) + \mathbf{J}_{\Phi}(x_*)\left(x_k-x_*\right)+O\left(\|x_k-x_*\|^2\right).
\end{equation*}
In fact, $\mathbf{J}_{\Phi}(x_*)=0$. Before presenting this result, we first show that 
\begin{equation}\label{theorem2.4}
\overline{C}(x_*)=\operatorname{diag}(g_*^{i})J^v(x_*)=\operatorname{diag}\bigg(\frac{\partial \left| v^i(x_*) \right|}{\partial x_*^i} \cdot g_*^i\bigg)=0.
\end{equation}
Case 1 ($x_*^i$ satisfies $l^i<x_*^i<u^i$).

By \eqref{first-order}, we immediately know that $g_*^{i}=0$, which leads to $\frac{\partial \left|v^i(x_*)\right|}{\partial x_*^i} \cdot g_*^i=0$ .\\
Case 2 ($x_*^i$ satisfies $x_*^i=u^i$ or $x_*^i=l^i$).

From what has been stated earlier, $J^v(x_*)$ is the Jacobian matrix of $|v(x)|$ at $x_*$, and based on the definition of the $v(x)$, we readily obtain that $\frac{\partial \left|v^i(x_*)\right|}{\partial x_*^i}=0$. This also means that $\frac{\partial \left|v^i(x_*)\right|}{\partial x_*^i} \cdot g_*^i=0$.\\
Through the combination of the two cases discussed above, we demonstrate that $\overline{C}(x_*)=0$, a conclusion that further implies 
\begin{equation}\label{theorem2.5}
\overline{B}(x_*)=D^{-1}(x_*)H(x_*)D^{-1}(x_*). 
\end{equation}
Now, we return to illustrate the fact that $\mathbf{J}_{\Phi}(x_*)=0$.
\begin{eqnarray*}
&&\mathbf{J}_{\Phi}(x_*) \\
&=& \frac{\partial}{\partial x}\left(x-D^{-1}(x)\overline{B}^{-1}(x) \overline{g}(x)\right)\bigg|_{x=x_*} \notag\\
&=&I-\frac{\partial}{\partial x}\left(D^{-1}(x)\overline{B}^{-1}(x) \overline{g}(x)\right)\bigg|_{x=x_*}\notag\\
&=&I-\frac{\partial (D^{-1}(x)\overline{B}^{-1}(x))}{\partial x}\bigg|_{x=x_*}\overline{g}(x_*)+ D^{-1}(x_*)\overline{B}^{-1}(x_*)\frac{\partial \overline{g}(x) }{\partial x}\bigg|_{x=x_*}\notag\\
&=&I-D^{-1}(x_*)\overline{B}^{-1}(x_*)\frac{\partial \overline{g}(x) }{\partial x}\bigg|_{x=x_*}\notag\\
&=&I-D^{-1}(x_*)\overline{B}^{-1}(x_*)\left(\frac{\partial D^{-1}(x)}{\partial x}\bigg|_{x=x_*} g(x_*)+D^{-1}(x_*)\frac{\partial g(x)}{\partial x}\bigg|_{x=x_*}\right)\notag\\
&=&I-D^{-1}(x_*)\overline{B}^{-1}(x_*)D^{-1}(x_*)H(x_*)\notag\\
&\overset{\rm{\eqref{theorem2.5}}}=&I-D^{-1}(x_*)\overline{B}^{-1}(x_*)\overline{B}(x_*)D(x_*)\notag\\
&=&0.
\end{eqnarray*}
Futhermore,
\begin{equation*}
\Phi(x_k) =x_k-D^{-1}(x_k)\overline{B}^{-1}(x_k) \overline{g}(x_k) = x_*+O\left(\|x_k-x_*\|^2\right).
\end{equation*}
Based on the analysis above, we obtain \eqref{theorem2.3}.

Recall that in lemma \ref{lemma13}, we have
\begin{equation*}
\|\overline{d}_k\|=\|-(\overline{B}_k+\theta_kI)^{-1}\overline{g}_k\|.
\end{equation*}
Moreover, when it is combined with $\overline{g}(x_*)=0$, the above equation becomes 
%\begin{eqnarry}
%\|\overline{d}_k\|&=&\|-(\overline{B}_k+\theta_kI)^{-1}\overline{g}_k+(\overline{B}_k+\theta_kI)^{-1}\overline{g}(x_*)\| \nonumber\\
%&\leq& \|(\overline{B}_k+\theta_kI)^{-1}\| \|\overline{g}(x_*)-\overline{g}_k\|\\
%&\leq& \frac{1}{\mu+\theta_k} \|D^{-1}_k {g}_k - D^{-1}(x_*){g}(x_*)\|\\
%&\leq& \frac{1}{\mu} \|D^{-1}_k {g}_k - D^{-1}(x_*){g}(x_*)\|\\
%&=& \frac{1}{\mu} \|D^{-1}_k {g}_k -D^{-1}_kg(x_*)+D^{-1}_kg(x_*)- D^{-1}(x_*){g}(x_*)\|\\
%&=& \frac{1}{\mu} \bigg[\|D^{-1}_k\| \|{g}_k -g(x_*)\|+ \|D^{-1}_k- D^{-1}(x_*)\| \|{g}(x_*)\|\bigg]\\
%&\leq& \frac{1}{\mu} \bigg[\gamma_D (\|H_k\|\|x_k -x_*\| + \frac{L_H}{2}\|x_k -x_*\|^2) +L_DU_g\|x_k- x_*\| \bigg]\\
%&\leq&\frac{1}{\mu} \bigg[\gamma_D (U_H\|x_k -x_*\| + \frac{L_HR}{2}\|x_k -x_*\|) +L_DU_g\|x_k- x_*\| \bigg],
%\end{eqnarry}\label{theorem2.4}
\begin{eqnarray*}
&&\|\overline{d}_k\| \\
&=& \| -(\overline{B}_k + \theta_k I)^{-1} \overline{g}_k + (\overline{B}_k + \theta_k I)^{-1} \overline{g}(x_*) \|  \\
&\leq& \| (\overline{B}_k + \theta_k I)^{-1} \| \| \overline{g}(x_*) - \overline{g}_k \| \\
&\leq& \frac{1}{\mu + \theta_k} \| D_k^{-1} g_k - D^{-1}(x_*) g(x_*) \|  \\
&\leq& \frac{1}{\mu} \| D_k^{-1} g_k - D^{-1}(x_*) g(x_*)\|  \\
&=& \frac{1}{\mu} \| D_k^{-1} g_k - D_k^{-1} g(x_*) + D_k^{-1} g(x_*) - D^{-1}(x_*) g(x_*) \|  \\
&\leq& \frac{1}{\mu} \left(\| D_k^{-1}\| \| g_k - g(x_*)\| + \| D_k^{-1} - D^{-1}(x_*)\| \| g(x_*) \| \right)  \\
& \overset{\rm{\eqref{lemma6.1a},\eqref{D.lipschitz}}} \leq& \frac{1}{\mu} \left( \gamma_D \left( \|H_k\| \|x_k - x_*\| + \frac{L_H}{2} \|x_k - x_*\|^2 \right) + L_D U_g \|x_k - x_*\| \right)  \\
&\leq& \frac{1}{\mu} \left( \gamma_D \left( U_H \|x_k - x_*\| + \frac{L_H r}{2} \|x_k - x_*\| \right) + L_D U_g \|x_k - x_*\| \right),
\end{eqnarray*}
where the last inequality holds since $x_k\in B(x_*,r)$ and $U_g=2U_H\gamma_D\delta\Delta$. Equivalently, we can obtain that
\begin{equation}\label{theorem2.6}
\|\overline{d}_k\| \leq O\left(\|x_k- x_*\|\right).
\end{equation}
It follows $x_{k+1}=x_k+D_k^{-1} \overline{d}_k$ from lemma \ref{lemma13}. Therefore,
\begin{eqnarray*}
 &&\|x_{k+1} - x_*\| \\
 & =& \|x_{k} +  D_k^{-1}\overline{d}_k + D_k^{-1}\overline{B}_k^{-1} \overline{g}_k - D_k^{-1}\overline{B}_k^{-1} \overline{g}_k - x_* \|  \\
 & \leq& \|x_{k} - D_k^{-1}\overline{B}_k^{-1} \overline{g}_k - x_*\| +  \|D_k^{-1}\overline{B}_k^{-1} \overline{g}_k + D_k^{-1}\overline{d}_k \|\\
 & \overset{\rm{\eqref{theorem2.2}}} \leq& \|x_{k} - D_k^{-1}\overline{B}_k^{-1} \overline{g}_k - x_*\| +  \Delta \gamma_D  \| \overline{d}_k \|^2\\
 & \overset{\rm{\eqref{theorem2.3},\eqref{theorem2.6}}} \leq& O\left(\|x_k-x_*\|^2\right).
\end{eqnarray*}
Thus, we obtain the desired result that the algorithm achieves local superlinear convergence rate.
\end{proof}

\section{Numerical Results}\label{section6}
In this section, we report the preliminary numerical results of SOBASIP (Algorithm \ref{SOBASIP}), which is implemented using MATLAB code and run under MATLAB(R2022b). The program runs on a desktop computer with the 11th Gen Intel(R) Core(TM) i5-1135G7 @ 2.40GHz   2.42 GHz.

During the algorithm implementation process, the parameters are selected as follows.

$\epsilon=10^{-6}$, $\nu=0,01$, $\delta=10^{-6}$, $\Delta=10^{-1}$, $\beta=0.5$, $\gamma=0.1$, $\tau=0.995$.

The numerical results of SOBASIP (Algorithm \ref{SOBASIP}) are presented in Table \ref{TABLE1}. where problems come from the testing problem set CUTEst \cite{gould26}. $N_{it}$, $N_f$, $N_g$ represent the number of iterations, the number of evaluations of $f(x)$, and the number of evaluations of $g(x)$, respectively. The columns $n$ is the number of variables.

\begin{center}
	{
		\begin{longtable}{l c c c c c c c}  
            \captionsetup{justification=raggedright,singlelinecheck=false}
			\multicolumn{8}{l}{\itshape Table \thetable{} Numerical experimental results of Algorithm \ref{SOBASIP}}\label{TABLE1}\\
            \hline 
			Problem  &$n$ &$N_{it}$ &$N_f$  &$N_g$ &$\|\overline{g}_k\|$ &$\lambda_1 (\overline{B}(x))$ &Cpu-time\\
			\hline
            \endfirsthead
%%            \multicolumn{6}{c}{TABLE \ref{TABLE1} continued}\\ 
%           \multicolumn{8}{l}{\itshape Table \thetable{} Numerical experimental results of Algorithm \ref{SOBASIP}}\\
            \hline
			Problem  &$n$ &$N_{it}$ &$N_f$  &$N_g$  &$\|\overline{g}_k\|$ &$\lambda_1 (\overline{B}(x))$ &Cpu-time \\
            \hline
            \endhead
            ALLINIT   &4 &6 &7 &7 &1.1518e-07  &1.0754e+01 &7.0796e-02 \\
%            BQPGASIM  &50 &1 &2 &2 & & &5.7322e-02 \\
            BQPIVAR   &1  &6 &7&7 &2.0736e-09  &1.0000e+00 &1.1669e-02 \\
            BIGGS5    &6  &20 &22 &21 & 4.6636e-07 &2.4072e-04 &1.6611e-01 \\
            BDEXP     &100  &30 &31 &31 & 8.9921e-07 &-7.7462e-08 &5.4854e-01 \\
            CAMEL6    &2   &6 &8 &7 &2.3270e-07 &2.2232e+01 &2.2502e-02\\
            HIMMELP1  &2    &10 &11 &11 &4.4474e-12  & 8.9953e+00 &3.3701e-02 \\
            HATFLDA   &4    &10 &11 &11 &7.3101e-07  &1.0489e-02 &3.8675e-02  \\
 %           HATFLDB   &4    &9 &10 &10 &1.1797e-01 \\
            HATFLDC   &25   &5 &6 &6 &5.5210e-12 &1.3333e+00 &6.2107e-02\\
            HS3MOD    &2    &22&23 &23 & 3.7215e-07 &1.0000e+00 &3.2903e-02 \\
            HS110     &50   &5 &6 &6 &6.6814e-13 & 4.8091e+01 &3.9415e-02  \\
            HS05      &2    &5 &6 &6 &1.5017e-07 &1.9329e+00 &9.0871e-01 \\
            HS25      &3    &16 &17 &17 &1.3808e-11  &6.1845e-04 &5.6836e-01 \\
            HS38      &4    &36 &50 &37 & 7.2335e-08 &7.9152e+00 &9.4834e-02\\
%            HS45      &5    &35 &36 &36 &4.2862e-01 \\
            JNLBRNGA  &16  &4 &5 &4 & 3.3079e-09 & 5.4742e-01 &2.8074e-02 \\
            JNLBRNGA  &100  &5 &6 &5 & 6.8249e-09&4.2717e-02 &3.3144e-01\\
            JNLBRNGA  &529  &6 &7 &6 &3.5165e-07 & 2.8141e-03 &1.8764e+00\\
            JNLBRNGA  &1024 &8 &9 &8 & 2.2888e-08 & 4.2473e-04 &5.2699e+00\\
            %JNLBRNGA  &5625 &13 &14 &13 &2.2121e+02\\
%            JNLBRNGA  &15625 &20 &21 &20 &3.2049e+03\\
            JNLBRNGB  &16  &6 &7 &6 & 2.7004e-07 &4.8544e+00 &2.9307e-02\\
            JNLBRNGB  &100  &11 &12 &11 & 3.8127e-09 &1.5563e-01 &1.2968e+00\\
            JNLBRNGB  &529  &21 &22 &21 & 7.5714e-07 &6.8365e-05 &8.8243e+00 \\
            JNLBRNGB  &1024  &28 &29 &28 & 7.0167e-08 &8.1481e-04 &4.9981e+01\\
            LINVERSE  &19  &17 &23 &18 &5.9162e-07 &-1.7191e-09 &5.3593e-01\\
            MCCORMCK  &10   &10 &11 &11 &1.3928e-08 &9.3466e-01 &1.8531e-01 \\
            MCCORMCK  &100   &17 &18 &18 &4.6360e-07 & 9.3466e-01 &4.2292e-01  \\
            NONSCOMP  &25   &27 &28 &28 & 4.6205e-07 &1.0954e-08 &2.2653e-01 \\
            OBSTCLAL  &100  &4 &5 &5 & 3.9018e-08 &2.8269e-02 &3.2353e-01\\
%            OBSTCLBL  &16   &3 &4 &4 &1.6282e-02 \\
%            OBSTCLBL  &100  &15 &16 &16 &8.4260e-01 \\
%            OBSTCLBL  &529  &28 &29 &29 &7.6546e+00 \\
%            OBSTCLBL  &1024  &29 &30 &30 &1.4808e+01\\
%            OBSTCLBU  &100  &35 &36 &36 &1.6107e+00 \\
            PROBPENL  &500  &8 &9 &9 & 3.1265e-06 &-1.1251e-06 &4.4423e+00\\
            PALMER1   &4    &20 &23 &20 &1.4806e-07 &-3.2480e-09 &4.9606e-01\\
            PALEMR2   &4    &19 &21 &20 & 3.7524e-08 &1.0230e-09 &2.3651e-01\\
            PALEMR3   &4    &61&65 &62 & 4.1870e-10 &6.6087e-12 &5.2469e-01  \\
            PALEMR4   &4    &81 &90&82& 4.1409e-08 & 1.9337e-10 &1.9247e-01 \\
            PSPDOC    &4    &10 &11 &11 & 4.1876e-09 &1.9806e-01 &3.2119e-02 \\
            SIMBQP    &2     &17 &18 &18 & 4.3017e-08 & 1.0000e+00 &3.0193e-02 \\
            S242     &3    &12&13 &13 &1.7379e-08 &-1.2226e-10 &5.4171e-02\\
            S328      &2   &6 &7 &7   &7.9221e-07 &2.1108e-01 &1.4485e-02 \\
%            S368      &8    &12 &13 &13 &2.2050e-01\\
%            TORSION1  &100  &8 &9 &9 &3.9894e-01\\
%            TORSION2  &100  &18 &19 &19 &8.3581-01\\
%            TORSION3  &100  &4 &5 &5 &1.8928e-01\\
%            TORSION6  &100  &13 &14 &14 &3.3813e-01\\
			\hline
		\end{longtable}
	}
\end{center}

\section{Conclusion}\label{section7}
In this paper, we extend the homogeneous second-order descent method for solving problems with bound constraints by proposing a second-order descent algorithm based on affine scaling interior-point methods. To remove bound constraints, an affine subproblem is constructed by using an affine matrix and the optimality conditions of problem \eqref{f}. 
Utilizing the homogenisation technique, an ordinary homogeneous model can be easily derived from the affine scaling subproblem, and it can be solved as an eigenvalue problem. The backtracking line search is employed to ensure a sufficient decrease of the objective function and further determine a new iteration point in Algorithm \ref{SOBASIP}.
Notably, the proposed algorithm attains a global iteration complexity of $O(\epsilon^{-3/2})$ for finding an $\epsilon$-approximate second-order stationary point, matching the iterative complexity of the HSODM proposed in \cite{zhang14}. Furthermore, the local convergence of SOBASIP is analysed under appropriate assumptions and the numerical results are reported indicating the practical viability of this approach.

\section*{Acknowledgments}
%The authors are very grateful to the editor and the referees, whose valuable suggestions and insightful comments helped to improve significantly the paper.\\
The authors are very grateful to the editor and the referees, whose valuable suggestions and insightful comments helped to improve significantly the paper. 

\section*{Funding}
Natural Science Foundation of Henan Province (252300421993); National Natural Science Foundation of China (12071133); Key Scientific Research Project for Colleges and Universities in Henan Province (25B110005).

%\bmhead{Data Availability} The manuscript does not contain any associated data.
\bibliographystyle{sn-mathphys-ay}
\bibliography{jieyueshu}
\end{document}